\documentclass[11pt]{article}
\usepackage{amsmath,amscd,amssymb,amsthm,array,colortbl,enumitem}
\usepackage{amsfonts}
\usepackage{amssymb}
\usepackage{amsmath}

\usepackage[english]{babel}
\usepackage[utf8]{inputenc}
\usepackage[T1]{fontenc}

\usepackage{tikz}
\usetikzlibrary{patterns}
\usepackage{stmaryrd}
\usepackage{dsfont}
\usepackage[normalem]{ulem}
\usepackage{setspace}
\usepackage{adjustbox}
\usepackage{calrsfs}

\usepackage{color,graphicx,psfrag}
\usepackage{amssymb,amsmath,amsthm,mathrsfs}
\usepackage{wrapfig}
\usepackage[mathscr]{eucal}

\usepackage{makeidx}
\usepackage{appendix}

\usepackage{mathabx}
\usepackage{slashbox}

\makeindex

\newcommand{\NN}{\mathbb{N}}
\newcommand{\ZZ}{\mathbb{Z}}

\newcommand{\dis}{\displaystyle}
\renewcommand{\epsilon}{\varepsilon}

\newcommand{\Supp}{\text{\rm supp}}
\newcommand{\Card}{\text{\rm card}}

\theoremstyle{definition}

\newtheorem{theorem}{Theorem}[section]
\newtheorem{lemma}[theorem]{Lemma}

\newtheorem{proposition}[theorem]{Proposition}
\newtheorem{definition}[theorem]{Definition}

\newtheorem{remark}[theorem]{Remark}

\title{Zero-Temperature Chaos in Bidimensional Models with Finite-Range Potentials}
\author{Sebasti\'an Barbieri$^1$, Rodrigo Bissacot$^{2,3}$, Gregório Dalle Vedove$^{2,4}$, Philippe Thieullen$^4$}
\date{%
	\small{$^1$Departamento de Matem\'aticas y Ciencia de la Computaci\'on, Universidad de Santiago de Chile \\
	$^2$Institute of Mathematics and Statistics, University of São Paulo \\
	$^3$Faculty of Mathematics and Computer Science, Nicolaus Copernicus University\\
	$^4$Institut de Mathématiques de Bordeaux, Université de Bordeaux}\\
	[2ex] 
	\begin{footnotesize}
	emails: sebastian.barbieri@usach.cl, rodrigo.bissacot@gmail.com; gregldvn@gmail.com; philippe.thieullen@math.u-bordeaux.fr
	\end{footnotesize}\\
	\today
	\vskip.1cm}

\begin{document}

\maketitle

\begin{abstract}
We construct a finite-range potential on a bidimensional full shift on a finite alphabet that exhibits a zero-temperature chaotic behavior as introduced by van Enter and Ruszel. This is the phenomenon where there exists a sequence of temperatures that converges to zero for which the whole set of equilibrium measures at these given temperatures oscillates between two sets of ground states. Br\'emont's work shows that the phenomenon of non-convergence does not exist for finite-range potentials in dimension one for finite alphabets; Leplaideur obtained a different proof for the same fact. Chazottes and Hochman
provided the first example of non-convergence in higher dimensions $d\geq3$; we extend their result for $d=2$ and highlight the importance of two estimates of recursive nature that are crucial for this proof: the relative complexity and the reconstruction function of an extension.

We note that a different
proof of this result was found by Chazottes and Shinoda, at around the same time that this article was initially submitted and that a strong generalization has been found by Gayral, Sablik and Taati.
\end{abstract}

\section{Introduction}

The states of a system at equilibrium in statistical mechanics are usually described by a family of probability measures indexed by an inverse temperature $\beta$ called Gibbs states. There are several ways in the literature to formalize the notion of Gibbs states; the most common definition in probability and mathematical physics literature is considering DLR equations, originally proposed by R. Dobrushin~\cite{Dob68}, and by O. Lanford and D. Ruelle~\cite{LanRue69}, this is the standard definition on textbooks of these areas, see \cite{bovier_2006, friedli_velenik_2017, Georgii1988}.

We adopt a more ergodic approach focusing on \textit{equilibrium measures}, which for regular enough potentials correspond to the translation invariant DLR measures, see Ruelle \cite{Ruelle2004}, Muir \cite{Muir_2011} and Keller \cite{Keller1998}. A discussion about when the several notions of Gibbsianness are equivalent (or not) can be found in  \cite{Capocaccia1976,  vanEnterFernadezSokal1993, Keller1998, Kimura, Meyerovitch2013, Muir_2011}.

The existence of an equilibrium measure for a continuous potential on the full shift comes from compactness. In the one-dimensional setting, the equilibrium measure is unique, whereas uniqueness does not necessarily hold in the two-dimensional case. Our main goal is to understand the behavior of the whole set of equilibrium measures as the temperature goes to zero, showing the existence of potentials with a chaotic behavior in dimension 2. Given a potential $\varphi$, for each inverse temperature $\beta$ consider $\mu_\beta$ an equilibrium state associated with the potential $\beta \varphi$, the existence of weak$^*$-accumulation points for the family of invariant Borel probability measures $(\mu_\beta)_{\beta\geq0}$  as $\beta\to+\infty$ is a trivial consequence of the Banach-Alaoglu theorem. The measures obtained as accumulation points of such families are particular cases of {\it minimizing measures} (or ground states) that we briefly recall. A minimizing measure $\mu_{min}$ is an invariant probability measure that satisfies
\[
\int\! \varphi \,d \mu_{min} = \bar\varphi \ \ \text{where} \ \ \bar\varphi := \inf \Big\{ \int \! \varphi \, d \mu : \mu \ \text{translation invariant} \Big\}.
\]
The real number $\bar \varphi$ is called the {\it ergodic minimizing value} or {ground-state energy}  of the potential $\varphi$.

The union of the support of minimizing measures is a compact invariant set, called the {\it Mather set}, that prescribes the behavior of the equilibrium measures at zero temperature. Many of the ideas in ergodic optimization and the terminology as ergodic minimizing value, minimizing measures and Mather set comes from the theory of Lagrangian dynamics in the continuous setting, see Mather~\cite{Mather1982}, Ma\~n\'e \cite{Mane1996}, Fathi \cite{Fathi1997,Fathi1997bis,Fathi2020}, and from Aubry-Mather theory in the discrete setting, Forni, Mather \cite{ForniMather1994}, Garibaldi, Lopes \cite{GaribaldiLopes2008}, Garibaldi, Thieullen \cite{GaribaldiThieullen2011}, Su, de la Llave \cite{SuDeLaLlave2012}, Sorrentino \cite{Sorrentino2016}. A thorough review of ergodic optimization is done in Jenkinson \cite{Jenkinson2018} in the one-dimensional setting.

For generic norm summable interactions on a finite alphabet, there exists a unique minimizing measure $\mu_{min}$ of uniquely ergodic support  (see Ruelle~\cite{Ruelle1969}). Some further properties of ground states for generic interactions are discussed by van Enter and Miekisz, see~\cite{VanEnterMiekisz2000}. Therefore in the case above, for any family of equilibrium measures $(\mu_\beta)_{\beta\geq0}$ we have $\mu_\beta \to \mu_{min}$ as $\beta\to+\infty$ . On the other hand, if there are at least two minimizing measures, the sequence $(\mu_\beta)_{\beta\geq0}$ might not converge. The notion of zero-temperature chaotic behavior was introduced by van Enter and Ruszel in the seminal paper \cite{vanEnterRuszel2007}, see also \cite{BCLMS2011} for a more detailed proof. Nowadays, there are many examples potentials with a non-trivial Mather set that are examples of zero-temperature chaotic behavior. In the multidimensional construction of Chazottes and Hochman~\cite{ChazottesHochman2010}, while the Mather set is highly complex, the potential is geometrically simple as it is obtained as the distance function to the Mather set. In the work of Coronel and Rivera-Letelier~\cite{CoronelRiveraLetelier2015}, the space is one-dimensional and the Mather set is quite simple (it may be equal to two ergodic measures with disjoint supports), but the potential is less explicit. In the construction of two of us and E. Garibaldi~\cite{BissacotGaribaldiThieullen2017}, the Mather set is reduced to two fixed points $\{0^\infty\} \cup \{1^\infty\}$, while the potential does not have finite-range, but only summable variations; a zero-temperature phase diagram is then obtained showing a relationship between zero-temperature chaotic behavior and cancellation of the Peierls barrier. A complete understanding of the low-temperature behavior and the ground state structure in two or more dimensions, even for finite-range interactions, seems currently out of reach. The Lipschitz condition in one dimension, though similar to a strong decay of interaction, is not enough to guarantee the convergence of the equilibrium measures. As the present paper illustrates, a finite-range condition exhibits a rich set of behaviors. Recent work~\cite{gayralsabliktaati2023} shows that recursion-theoretic tools play an important role in the description of said structure.

In the discussion that follows, we shall restrict ourselves to the class of finite-range potentials. In the one-dimensional setting, the Mather set of a finite-range potential could have a rich structure of minimizing measures. It is a remarkable result that in this setting, the zero-temperature limit of Gibbs measures always exists and selects a particular minimizing measure that is not necessarily ergodic. This result was originally proven by Br\'emont~\cite{Bremont}, and was later given other proofs by Chazottes, Gambaudo and Ugalde~\cite{ChazottesGambaudoUgalde2011}, Leplaideur~\cite{Leplaideur2005}, and by Garibaldi and Thieullen~\cite{GaribaldiThieullen2011}, who also provide an algorithm that identifies the limiting minimizing measure. In the one-dimensional setting, for finite-range potentials, the Mather set is reduced to a finite disjoint union of subshifts of finite type (including periodic orbits), and the limiting minimizing measure is some barycenter of the measures of maximal topological entropy of these subshifts. The extension of Br\'emont's results to a countable alphabet has been undertaken by Jenkinson, Mauldin and Urb\'anski~\cite{JenkinsonMauldinUrbanski2005}, Morris~\cite{Morris2007bis}, Kempton~\cite{Kempton2011} for the BIP case, and recently the transitive case in \cite{BLMV}.

The status of the zero-temperature chaotic behavior for finite-range potentials in higher dimensions, $d\geq2$, is completely different. Chazottes and Hochman in~\cite{ChazottesHochman2010} constructed for every $d\geq3$ an example of a finite-range potential exhibiting a zero-temperature chaotic behavior. The dimension in their result needs to be greater or equal to $3$ because the proof relies heavily on a theorem of Hochman~\cite{Hochman} which realizes any one-dimensional effective dynamical systems as the topological factor of the subaction of a $\ZZ^3$-subshift of finite type. After this result, the only case missing was $d=2$. Our main result is an extension of their results for dimension 2.

\begin{theorem}
\label{main.theorem.generalized}
There exists a finite alphabet $\mathcal{A}$ and a finite-range potential $\varphi$ on a bidimensional full shift that exhibits the phenomenon of zero-temperature chaotic behavior.
\end{theorem}

Our construction is based on the simulation theorem of Aubrun and Sablik~\cite{AubrunSablik2013} which states that every one-dimensional effectively closed subshift, extended vertically trivially to a $2$-dimensional subshift, is a topological factor of a subshift of finite type of zero topological entropy. We remark that this result was simultaneously proven by Durand, Romashchenko, and Shen~\cite{DRS1, DRS2}, which is the main tool used by Chazottes and Shinoda~\cite{ChazottesShinoda2020} in their alternative proof.

While quite intricate, the extension constructed by Aubrun and Sablik has the advantage of being quite explicit, whereas the proof by Durand, Romashchenko, and Shen is based on Kleene's fixed point theorem. We use the Aubrun-Sablik construction to produce a few estimates which are not explicit in~\cite{ChazottesHochman2010}. These estimates provide bounds that control the relative complexity of the SFT extension. More precisely, we give an explicit bound of the reconstruction function of the extension, thus avoiding the need to use Kleene's fixed point theorem.

The outline of the proof is the following. In the second section, we give the main definitions, outline the strategy's main ideas, and give the proof of the Theorem \ref{main.theorem.generalized} assuming a number of estimates that arise from a yet unspecified construction. In the third section we explain the detailed construction of the one-dimensional subshift. In the fourth section the prove the estimates pertaining the bounds for the topological entropy. In the fifth section we prove the two estimates on the reconstruction function and complexity function in the Aubrun-Sablik simulation theorem.

The present paper is part of the thesis of the third author Greg\'orio Dalle Vedove. A preliminary version was submitted to arxiv~\cite{DalleVedove2020} at about the same time when a paper of Chazottes and Shinoda~\cite{ChazottesShinoda2020} was submitted proving the same result but with a different proof.

We remark that recently, in a beautiful paper of Gayral, Sablik and Taati~\cite{gayralsabliktaati2023}, the authors obtain a recursion-theoretic classification of the set of ground states for \textit{computable} finite range interactions up to computable affine homeomorphisms. Their result not only implies ours and that of Chazottes and Shinoda, but provides a strong argument that recursion theory is not only a tool to obtain zero-temperature chaotic behavior, but that in fact is a central piece of the puzzle needed to understand the possible sets of ground states.

\section{Definitions and outline of the proof} \label{section:OutlineProof}

We summarize our setting in the following definitions. 

\begin{definition}
\label{def.shift-action}
Let $\mathcal{A}$ be a finite set called {\it alphabet} and $d\geq1$ an integer. The space of $d$-dimensional configurations $\Sigma^d(\mathcal{A}) = \mathcal{A}^{\ZZ^d}$ is the {\it d-dimensional full shift}. The {\it shift action} is the $\ZZ^d$-action given $\sigma = (\sigma^u)_{u\in\ZZ^d}$, $\sigma^u\colon \Sigma^d(\mathcal{A})\to\Sigma^d(\mathcal{A})$ defined by
 \[
\sigma^u(x)=y \ \mbox{ if } y(v) = x(u+v) \mbox{ for every } x,y \in\Sigma^d(\mathcal{A}) \mbox{ and } v \in \ZZ^d.
\]

\end{definition}

We recall that an invariant probability measure $\mu$ for the $\mathbb{Z}^d$ action is a Borel measure on $\Sigma^d(\mathcal{A})$ such that for every Borel set $B$ we have
$ \mu(\sigma^u(B)) = \mu(B) \mbox{ for every }u\in\mathbb{Z}^d$.

The set of invariant probability measures is denoted by $\mathcal{M} = \mathcal{M}(\Sigma^d(\mathcal{A}),\sigma)$.

In this article we choose a function $\varphi : \Sigma^2(\mathcal{A}) \to \mathbb{R}$ that is supposed to describe the energy contribution at the origin of the lattice $\mathbb{Z}^d$. A {\it potential} is a function $\varphi \colon \Sigma^d(\mathcal{A}) \to \mathbb{R}$.

\begin{definition}
Let $\varphi : \Sigma^d(\mathcal{A}) \to \mathbb{R}$ be a Lipschitz function.
\begin{enumerate}
\item The {\it pressure} of the potential $\varphi$ is the real number
\[
P(\varphi ) := \sup_{\mu \in\mathcal{M}} \Big\{ h(\mu)-\int\! \varphi \, d\mu \Big\}.
\]
where  $h(\mu)$ denotes the {\it Kolmogorov-Sinai  entropy} of $\mu$ (definition \ref{definition:EntropyDefinition}).  
\item  An {\it equilibrium measure at inverse temperature $\beta$}, is an invariant probability measure $\mu_\beta$ that maximizes the pressure
\[
P(\beta \varphi ) = h(\mu_\beta) - \int \! \beta  \varphi \,d \mu_\beta.
\]
\end{enumerate}
\end{definition}

The set of equilibrium measures at inverse temperature $\beta$ is denoted by $\mathcal{M}_{e}(\beta \varphi)$.

The general strategy follows the outlines presented in van Enter and Ruszel~\cite{vanEnterRuszel2007}, Coronel and Rivera-Letelier~\cite{CoronelRiveraLetelier2015}, and Chazottes and Hochman~\cite{ChazottesHochman2010}. 

\begin{definition} \label{definition:ChaoticConvergence}
The phenomenon of {\it zero-temperature chaotic behavior} holds for a potential $\varphi : \Sigma^d(\mathcal{A}) \to \mathbb{R}$ when there exists a sequence of inverse temperatures $(\beta_k)_{k\geq0}$ going to infinity and  two disjoint compact sets ${\widetilde G_1}$ and ${\widetilde G_2}$ of $\Sigma^d(\mathcal{A})$ such that, if for every $k\geq0$, $\mu_{\beta_k}$ is any choice of an ergodic equilibrium measure $\mu_{\beta_k} \in \mathcal{M}_{e}(\beta_k \varphi)$, the support of any weak${^*}$-accumulation point of the odd subsequence $(\mu_{\beta_{2k+1}})_{k\geq0}$, respectively of the even subsequence  $(\mu_{\beta_{2k}})_{k\geq0}$, is included in ${\widetilde G_1}$, respectively in ${\widetilde G_2}$.
\end{definition}

We shall restrict ourselves to the class of {\it finite-range} potentials.

\begin{definition}
Let $\mathcal{A}$ be a finite set and $D\geq1$ be an integer. A function $\varphi\colon \Sigma^d(\mathcal{A}) \to \mathbb{R}$ is {\it finite-range (of range $D)$} if
\[
 x|_{\llbracket 1,D\rrbracket^d} = y|_{\llbracket 1,D \rrbracket^d} \ \Rightarrow \ \varphi(x) = \varphi(y)  \mbox{ for every }x,y \in \Sigma^d(\mathcal{A}),
\]
where $ x|_{\llbracket 1,D\rrbracket^d}$ denotes the restriction of a configuration $x$ to the square $\llbracket 1,D\rrbracket^d$.
\end{definition}

We recall now several definitions.

\begin{definition}
Let $\mathcal{A}$ be a finite set.
\begin{enumerate}
\item If $S$ is a finite subset of $\mathbb{Z}^d$, a {\it pattern with support $S$} is a partial configuration $p \in \mathcal{A}^{S}$. The set $S=\Supp(p)$ is called the {\it support} of $p$. For $d=1$ a pattern is called a {\it word}. If $x \in \Sigma^d(\mathcal{A})$, $p=x|_S$ denotes the pattern obtained by taking the restriction of $x$ to $S$. A pattern $p$ of size $n$ is a pattern of the form $p \in \mathcal{A}^{\llbracket 1, n \rrbracket^d}$.
\item The shift action extends to an action over the set of patterns, that is, for every $u \in \mathbb{Z}^d$ and $p \in \mathcal{A}^S$ we say that $\sigma^u(p)=p'$ if and only if $p' \in \mathcal{A}^{S-u}$ and for every $v \in S-u$ we have $p'(v)=p(u+v)$.
%\[
%\forall\, u \in \mathbb{Z}^d, \ \forall\, p \in \mathcal{A}^S, \ \sigma^u(p)=p' \ \Leftrightarrow \ p' \in \mathcal{A}^{S-u} \ \ \text{and} \ \ \forall\,  v \in S-u, \ p'(v)=p(u+v). 
%\]
\item A pattern $p\in\mathcal{A}^S$ {\it appears} in a configuration $x\in\Sigma^d(\mathcal{A})$ if there exists  $u\in\mathbb{Z}^d$ such that $\sigma^u(x)|_{S}=p$. We write $p \sqsubset x$. More generally a pattern $p\in\mathcal{A}^S$ {\it appears} in a pattern $q\in \mathcal{A}^T$ if there exists $u\in\mathbb{Z}^d$ such that $S \subseteq T-u$ and $\sigma^u(q)|_{S}=p$.
\end{enumerate}
\end{definition}

\begin{definition}
\label{def.subshift}
A {\it subshift} $X$ is a closed subset of $\Sigma^d(\mathcal{A})$ which is invariant by the shift action $\sigma$.
\end{definition}

Subshifts can also be given a convenient combinatorial description by exhibiting a set of forbidden patterns, that is, let $\mathcal{F} \subseteq \bigsqcup_{n\geq1} \mathcal{A}^{\llbracket 1,n \rrbracket^d}$ and consider the set of all $x \in \Sigma^d(\mathcal{A})$ such that $p \not\sqsubset x$ for every pattern $p \in \mathcal{F}$, then it follows that this set is closed and invariant under the shift action. This motivates the following definition.

\begin{definition}
Let $\mathcal{A}$ be a finite set and $\mathcal{F} \subseteq \bigsqcup_{n\geq1} \mathcal{A}^{\llbracket 1,n \rrbracket^d}$ be a set of patterns called {\it the set of forbidden patterns}. A subshift $X \subseteq \Sigma^d(\mathcal{A})$ is said to be {\it generated by $\mathcal{F}$} if for every $x \in X$ and $p \in \mathcal{F}$, $p$ does not appear in $x$. More formally
\[
X=\Sigma^d(\mathcal{A},\mathcal{F}) :=  \big\{ x \in \Sigma^d(\mathcal{A}) : \forall\, p \in \mathcal{F}, \ p \not\sqsubset x \big\}. 
\]
A subshift $X=\Sigma^d(\mathcal{A},\mathcal{F})$ is said to be a {\it subshift of finite type} (or SFT) if the set of forbidden patterns $\mathcal{F}$ is a finite set.
\end{definition}

It is clear that every subshift is generated by some set of forbidden patterns, namely, $X$ is generated by the full set of forbidden patterns $\mathcal{F}_X$ where $p \in \mathcal{F}_X$ if and only if $p$ does not appear in any $x \in X$. However, we remark that a fixed subshift $X$ can be generated by different sets of forbidden patterns. The first step of our construction consists in choosing a 1-dimensional subshift $\widetilde X \subseteq \Sigma^1(\widetilde{\mathcal{A}})$ in the following way. The alphabet $\widetilde{\mathcal{A}}$ is made of two alphabets $\widetilde{\mathcal{A}} = \widetilde{\mathcal{A}}_1 \cup \widetilde{\mathcal{A}}_2$ where
\[
\widetilde{\mathcal{A}}_1 = \{0,1\} \quad\mbox{and}\quad \widetilde{\mathcal{A}}_2 = \{0,2\}.
\]
The subshift $\widetilde X = \bigcap_{k\geq0}\widetilde X_k$ is obtained as the intersection of a decreasing sequence of subshifts $\widetilde X_k$ of controlled complexity
\begin{gather*}
\widetilde X_{k+1} \subseteq \widetilde X_k.
\end{gather*}
Each $\widetilde X_k$ contains a disjoint union of two  subshifts 
\begin{gather*}
\widetilde X_{k}^A \sqcup \widetilde X_{k}^B \subseteq \widetilde X_k.
\end{gather*}
The subshift $\widetilde X_k^A$ consists of configurations over the symbols $\{0,1\}$ and satisfies,
\begin{gather*}
\{ 1^\infty \} \subset \widetilde X_{k+1}^A \subseteq \widetilde X_{k}^A \subseteq \Sigma^1(\widetilde{\mathcal{A}}_1), \quad \widetilde X^A :=  \bigcap_{k\geq0}\widetilde X_{k}^A.
\end{gather*}
The subshift $\widetilde X_k^B$ consists of configurations over the symbols $\{0,2\}$ and satisfies,
\begin{gather*}
 \{2^\infty\} \subset \widetilde X_{k+1}^B \subseteq \widetilde X_{k}^B \subseteq \Sigma^1( \widetilde{\mathcal{A}}_2), \quad \widetilde X^B :=  \bigcap_{k\geq0} \widetilde X_{k}^B.
\end{gather*} 
The two subshifts $\widetilde X_{k}^A$ and $\widetilde X_{k}^B$ are chosen so that their relative complexity alternates depending on whether $k$ is odd or even. We use the symbol $0 \in \widetilde{\mathcal{A}}$ to measure the complexity or the frequency of $0$ in each word of $\widetilde X_{k}^A$ or $\widetilde X_{k}^B$. We finally make sure that $\widetilde X$ is effectively closed as in the following definition. In what follows, we use the definition of {\it Turing machine} as in Sipser~\cite[Definition 3.3]{Sipser}.

\begin{definition} \label{Definition:TimeEnumerationFunction}
A subshift $\widetilde X \subseteq \Sigma^1(\widetilde{\mathcal{A}})$ is said to be effectively closed if there exists a set of forbidden words $\widetilde{\mathcal{F}} \subseteq \bigsqcup_{n\geq1} \widetilde{\mathcal{A}}^{\llbracket 1,n \rrbracket}$ such that $\widetilde X = \Sigma^1(\widetilde{\mathcal{A}},\widetilde{\mathcal{F}})$ and  $\widetilde{\mathcal{F}}$  is enumerated by a Turing machine $\widetilde{\mathbb{M}}$.   
The {\it time enumeration function} $T^{\widetilde X} \colon \mathbb{N}_* \to \mathbb{N}_*$ associated to $\widetilde{\mathbb{M}}$ is given for $n\geq 1$ as the smallest positive integer $T^{\widetilde X}(n)$ such that $\widetilde{\mathbb{M}}$ halts on every word of $\widetilde{\mathcal{F}}$ of size at most $n$ in at most $T^{\widetilde X}(n)$ steps.
\end{definition}

In a second step of the construction we use the Aubrun-Sablik simulation theorem. Their result states that for every effectively closed subshift, one can find a two-dimensional subshift of finite type whose restriction to a one-dimensional subaction is a topological extension of the original effectively closed subshift. More precisely

%In a second step of the construction we use the Aubrun-Sablik simulation theorem. This result states that for every effectively closed subshift, there exists a new finite alphabet $\widehat{\mathcal{A}}$, a 2-dimensional subshift of finite type of zero topological entropy $\widehat X \subseteq \Sigma^2(\widehat{\mathcal{A}})$ generated by a finite set of forbidden words $\widehat{\mathcal{F}} \subseteq \widehat{\mathcal{A}}^{\llbracket 1,D\rrbracket^2}$ for some $D\geq1$, a surjective continuous map $\Sigma^2(\widehat{\mathcal{A}}) \to \Sigma^1(\widetilde{\mathcal{A}})$ that sends $\widehat X$ onto $\widetilde X$, a surjective group morphism $\mathbb{Z}^2 \to \mathbb{Z}$, such that  the $\mathbb{Z}^2$ action on $\Sigma^2(\widehat{\mathcal{A}})$ commutes with the $\mathbb{Z}$ action on $\Sigma^1(\widetilde{\mathcal{A}})$. More precisely

\begin{theorem}[Aubrun-Sablik \cite{AubrunSablik2013}] \label{theorem:AubrunSablikSimulation}
Let $\widetilde{\mathcal{A}}$ be a finite set and $\widetilde X =\Sigma^1(\widetilde{\mathcal{A}},\widetilde{\mathcal{F}})$ be an effectively closed subshift. There exists a finite alphabet $\mathcal{B}$, a subshift of finite type of zero topological entropy $\widehat X = \Sigma^2(\widehat{\mathcal{A}},\widehat{\mathcal{F}})$, where $\widehat{\mathcal{A}} = \widetilde{\mathcal{A}}\times\mathcal{B}$, $\widehat{\mathcal{F}} \subseteq \widehat{\mathcal{A}}^{\llbracket 1,D \rrbracket^2}$ for some integer $D\geq2$, such that if $\widehat\pi : \widehat{\mathcal{A}} \to \widetilde{\mathcal{A}}$ is the first projection and $\widehat\Pi : \Sigma^2(\widehat{\mathcal{A}}) \to \Sigma^2(\widetilde{\mathcal{A}})$ is the 1-block factor map defined component wise by
\[
 \widehat\Pi(x) = (\widehat\pi(x_u))_{u\in\mathbb{Z}^2} \mbox{ for every } x = (x_u)_{u\in\mathbb{Z}^2} \in \Sigma^2(\widehat{\mathcal{A}}),
\]
then $\widehat\Pi(\widehat X)$ simulates $\widetilde X$ in the sense $\widehat\Pi(\widehat X)  = \widetilde{\widetilde X}$ where
\[
\widetilde{\widetilde X} :=  \big\{ \widetilde x \in \Sigma^2(\widetilde{\mathcal{A}}) : \forall\, v \in \mathbb{Z}, \ (\widetilde x_{(u,v)})_{u \in\mathbb{Z}} = (\widetilde x_{(u,0)})_{u \in\mathbb{Z}} \in \widetilde X \big\}.
\]
The subshift $\widehat X$ will be called later the {\it Aubrun-Sablik SFT simulating $\widetilde X$}. The subshift $\widetilde{\widetilde X}$ will be called later the {\it vertically aligned subshift replicating $\widetilde X$}.
\end{theorem}

The remarkable fact is that, although the initial set of forbidden words $\widetilde{\mathcal{F}}$ might be countably infinite, by adding different ``colors'' $\mathcal{B}$ to the initial alphabet $\widetilde{\mathcal{A}}$ and by imposing a finite set $\widehat{\mathcal{F}}$ of forbidden rules on these juxtaposed colors, the new set of configurations $\widehat X$ which respects these rules, after applying the projection $\widehat{\Pi}$, describes exactly the set of  vertically aligned configurations of $\widetilde X$.

Our proof requires two a priori estimates, see inequalities (\ref{estimate_1}) and (\ref{estimate_2}), that were not stated explicitly in \cite{ChazottesHochman2010}. We recall first several definitions.

\begin{definition} \label{definition:ReconstructionFunction}
Let $\mathcal{A}$ be a finite alphabet,  $\mathcal{F} \subseteq \bigsqcup_{n\geq1}\mathcal{A}^{\llbracket 1,n\rrbracket^d}$ be a  set of forbidden patterns, and $X=\Sigma^d(\mathcal{A},\mathcal{F})$. 
\begin{enumerate}
\item A pattern $w \in \mathcal{A}^S$  is said to be {\it locally $\mathcal{F}$-admissible} if no pattern $p$ of $\mathcal{F}$ appears in $w$.
\item A pattern $w \in \mathcal{A}^S$ is said to be {\it globally $\mathcal{F}$-admissible} if $w$ appears in some configuration $x \in \Sigma^d(\mathcal{A},\mathcal{F})$.
\item The {\it reconstruction function} of $X$ is the function $R^X\colon \mathbb{N}_* \to \mathbb{N}_*$ such that, if $n\geq1$, then $R^X(n)$ is the smallest integer $R\geq n$ such that, for every locally $\mathcal{F}$-admissible pattern $p \in \mathcal{A}^{\llbracket -R, R \rrbracket^d}$, the subpattern $p|_{\llbracket -n,n \rrbracket^d}$ is globally $\mathcal{F}$-admissible.
\end{enumerate}
\end{definition}

By a standard compactness argument, it follows that every subshift admits a well defined reconstruction function. Our proof relies on the fact that the reconstruction function of the Aubrun-Sablik SFT increases at most exponentially for a particular choice of the initial 1-dimensional set of forbidden words $\widetilde{\mathcal{F}}$.

The first a priori estimate is given in the Inequality (\ref{estimate_1}) of the following proposition:

\begin{proposition} \label{proposition:ReconstructionFunctionEstimate}
Let $\widetilde{\mathcal{A}}$ be a finite set, $\widetilde{\mathcal{F}} = \bigsqcup_{n\geq1}\widetilde{\mathcal{F}_n}$, $\widetilde{\mathcal{F}_n} \subseteq  \widetilde{\mathcal{A}}^{\llbracket 1,n \rrbracket}$ be a set of forbidden words enumerated by a Turing machine $\widetilde{\mathbb{M}}$, $T^{\widetilde X}$ be the time enumeration function associated to $\widetilde{\mathbb{M}}$ (Definition \ref{Definition:TimeEnumerationFunction}), and $\widehat X$ be the Aubrun-Sablik SFT simulating $\widetilde X$ given in  \ref{theorem:AubrunSablikSimulation}. We assume that $\widetilde{\mathcal{F}}$ satisfies
:
\begin{enumerate}
\item \label{item:ReconstructionFunctionEstimate_1} $\widetilde{\mathbb{M}}$ enumerates all patterns of $\widetilde{\mathcal{F}}$ in increasing order (words of $\widetilde{\mathcal{F}}_n$ are enumerated before those in $\widetilde{\mathcal{F}}_{n+1}$). 
\item \label{item:ReconstructionFunctionEstimate_2} $R^{\widetilde X}(n) \leq Cn$, for some constant $C$,
\item \label{item:ReconstructionFunctionEstimate_3} $T^{\widetilde X}(n) \leq P(n)|\widetilde{\mathcal{A}}|^n$, for some polynomial $P(n)$.
\end{enumerate}
Then
\begin{equation}\label{estimate_1}
  \limsup_{n\to+\infty} \frac{1}{n}\log(R^{\widehat X}(n)) < +\infty.  
\end{equation}
\end{proposition}

The second a priori estimate, see Inequality (\ref{estimate_2}) of the next proposition, improves the computation of the complexity of the Aubrun-Sablik extension.
It is well know that both the vertically aligned subshift $\widetilde{\widetilde X}$ and the Aubrun-Sablik SFT $\widehat X$ have zero topological entropy. We actually need a stronger notion of complexity.

We recall first several definitions.

\begin{definition} \label{definition:Language}
Let $X\subseteq \Sigma^d(\mathcal{A})$ be a subshift. 
\begin{enumerate}
\item The {\it language of size $n$ of  $X$} is the set of  patterns of size $n$ that appear in $X$
\[
\mathcal{L}(X, n) :=  \Big\{ p \in \mathcal{A}^{\llbracket 1,n \rrbracket^d} :  \exists x \in X, \ \text{s.t.} \ p = x|_{\llbracket 1, n\rrbracket^d} \Big\}.
\]
\item The {\it language of $X$} is the disjoint union of languages of size $n$
\[
\mathcal{L}(X) := \bigsqcup_{n\geq1} \mathcal{L}(X,n).
\]
\end{enumerate}
\end{definition}

\begin{definition}
Let $\widetilde{X}=\Sigma^1(\widetilde{\mathcal{A}},\widetilde{\mathcal{F}})$ be an effectively closed subshift, and $\widehat{X} =\Sigma^2(\widehat{\mathcal{A}},\widehat{\mathcal{F}})$ be the Aubrun-Sablik  SFT given in theorem \ref{theorem:AubrunSablikSimulation} that simulates $\widetilde{X}$. The {\it relative complexity function} of the simulation is the function $C^{\widehat{X}} \colon \mathbb{N}_* \to \mathbb{N}_*$ defined by
\[
C^{\widehat{X}}(n):=\sup_{\widetilde{w}\in\mathcal{L}(\widehat\Pi(\widehat X),n)} \Card\big( \big\{ \widehat{w}\in\mathcal{L}(\widehat{X},n):\widehat{\Pi}(\widehat{w})=\widetilde{w} \big\}\big).
\]
\end{definition}

By construction the topological entropy of the simulating SFT is zero. This implies that
\[
\lim_{n\to+\infty} \frac{1}{n^2} \log(C^{\widehat X}(n)) =0.
\]
We prove a stronger result.

\begin{proposition} \label{proposition:ComplexityFunctionEstimate}
Let $\widetilde{X}=\Sigma^1(\widetilde{\mathcal{A}},\widetilde{\mathcal{F}})$ be an effectively closed subshift as described in proposition  \ref{proposition:ReconstructionFunctionEstimate},   and $\widehat{X} =\Sigma^2(\widehat{\mathcal{A}},\widehat{\mathcal{F}})$ be the Aubrun-Sablik  SFT given in theorem \ref{theorem:AubrunSablikSimulation} that simulates $\widetilde{X}$. Then
\begin{equation}\label{estimate_2}
\limsup_{n\to+\infty} \frac{1}{n} \log(C^{\widehat X}(n)) < +\infty.
\end{equation}
\end{proposition}
In the third and last step of the construction we enlarge the alphabet $\widehat{\mathcal{A}}$ by duplicating randomly the symbol $0$. Let
\begin{gather*}
\widetilde{\widetilde{\mathcal{A}}} = \{0',0'',1,2\} \quad\mbox{and}\quad \mathcal{A} = \widetilde{\widetilde{\mathcal{A}}} \times \mathcal{B}.
\end{gather*}
Let $\gamma \colon \mathcal{A} \to \widehat{\mathcal{A}}$ be the map that collapses $0'$ and $0''$ to $0$, $\Gamma \colon \Sigma^2(\mathcal{A}) \to \Sigma^2(\widehat{\mathcal{A}})$ be the corresponding 1-block factor map defined component-wise and extended to patterns, and
\begin{gather}\label{equation_1}
\mathcal{F} := \big\{ p \in \mathcal{A}^{\llbracket 1,D \rrbracket^2} : \Gamma(p) \in \widehat{\mathcal{F}} \big\}, \quad X := \Sigma^2(\mathcal{A},\mathcal{F}) = \big\{ x \in \Sigma^2(\mathcal{A}) : \Gamma(x) \in \widehat X \big\}.
\end{gather}
The subshift $X$ will be called thereafter the {\it duplicating SFT}. The composition maps $\widehat\pi \circ \gamma$ and $\widehat{\Pi} \circ \Gamma$ are denoted by
\begin{gather}\label{pi_functions}
\pi := \widehat\pi \circ \gamma : \mathcal{A} \to \widetilde{\mathcal{A}} \quad\mbox{and}\quad  \Pi := \widehat{\Pi} \circ \Gamma : \Sigma^2(\mathcal{A}) \to \Sigma^2(\widetilde{\mathcal{A}}).
\end{gather}
Notice that $X$ is a SFT generated by the finite set of forbidden patterns $\mathcal{F}$. 

The finite-range potential $\varphi$ of Theorem $1.1$  responsible for the zero-temperature chaotic behavior phenomenon may now be defined.
\begin{definition} \label{Definition:ShortRangePotential}
Let $X = \Sigma^2(\mathcal{A},\mathcal{F})$ be the duplicating SFT. The finite-range potential $\varphi : \Sigma^2(\mathcal{A}) \to \mathbb{R}$ is the function
\[
\varphi(x) =
\begin{cases}
1 \ \mbox{ if } \ x|_{\llbracket 1,D \rrbracket^2} \in\mathcal{F}, \\
0 \ \mbox{ if } \ x|_{\llbracket 1,D \rrbracket^2} \not\in\mathcal{F}.
\end{cases} \mbox{ for every } x \in \Sigma^2(\mathcal{A}).
\]
\end{definition}
Notice that $X = \{ x \in \Sigma^2(\mathcal{A}) :  \varphi \circ \sigma^u(x) = 0 \mbox{ for every } u \in \mathbb{Z}^2 \}$. In particular, the ergodic minimizing value $\bar\varphi$ of $\varphi$ is zero and the Mather set is the support of the set of invariant probability measures supported by $X$.
\[
\bar\varphi = 0 \ \ \text{and} \ \  \text{\rm  Mather}(\varphi) \subseteq X.
\]
A consequence is that any weak${}^*$ accumulation point of $(\mu_{\beta\varphi})_{\beta\to+\infty}$ must be a measure supported in $X$. 

The finite-range potential is the characteristic function of a cylinder set. We recall several definitions.

\begin{definition}
Let $\mathcal{A}$ be a finite alphabet, $a\in\mathcal{A}$ be a symbol, $S \subseteq \mathbb{Z}^d$ be a subset, $p\in \mathcal{A}^S$ be a pattern of support $S$, and $P\subseteq \mathcal{A}^S$ be a subset of patterns.
\begin{enumerate}
\item The {\it cylinder generated by  $a$}, denoted by $[a]_0$, is the set of configurations 
\[
[a]_0 = \{ x \in \Sigma^d(\mathcal{A}) : x(0)=a \}.
\] 
\item The {\it cylinder generated by  $p$}, denoted by $[p]$, is the set of configurations
\[
[p] := \{x \in \Sigma^d(\mathcal{A}) : x|_S = p \}.
\]
\item The {\it cylinder generated by $P$}, denoted by $[P]$, is the set of configurations 
\[
[P] := \bigsqcup_{p\in P}[p].
\]
\end{enumerate}
\end{definition} 

The finite-range potential is thus the characteristic function of the cylinder of forbidden words $\mathcal{F}$ 
defined in (\ref{equation_1})
\[
\varphi= \mathds{1}_{[\mathcal{F}]} : \Sigma^2(\mathcal{A}) \to \mathbb{R}.
\]

Our next task consists in describing the construction of the intermediate subshifts $\widetilde X_k$, $\widetilde X_{k}^A$, $\widetilde X_{k}^B$. To this end, we shall introduce the following notations.

\begin{definition}
Let $\mathcal{A}$ be a finite alphabet, and $d\geq1$ be an integer.
\begin{enumerate}
\item A {\it dictionary of size $\ell$ in dimension $d$} is a subset $L$ of patterns of $\mathcal{A}^{\llbracket 1,\ell \rrbracket^d}$. 
\item The {\it concatenated subshift} of a dictionary $L$ of size $\ell$  is the subshift of the form
\begin{equation*}
\label{eq.concatenated-subshift}
\begin{array}{rcl}
\langle L \rangle & = & \dis \bigcup_{u \in \llbracket 1,\ell \rrbracket^d} \bigcap_{v \in \mathbb{Z}^d} \sigma^{-(u+v \ell)} [L], \\
& = & \dis \Big\{ x \in \Sigma^d(\mathcal{A}) : \exists u \in \llbracket 1, \ell \rrbracket^d, \ \forall v \in \mathbb{Z}^d, \ (\sigma^{u+\ell v}(x) )|_{\llbracket 1,\ell \rrbracket^d} \in L \Big\}. \\
\end{array}
\end{equation*}
\end{enumerate}
\end{definition}

We construct by induction two sequences of dictionaries in dimension $1$, $(\widetilde A_k)_{k\geq0}$ and $(\widetilde B_k)_{k\geq0}$ using the alphabets $\widetilde{\mathcal{A}}_1$ and $\widetilde{\mathcal{A}}_2$ respectively in the following way. We choose a sequence of integers $(N_k)_{k\geq0}$, with $N_k \geq 4$ and define by induction the size $\ell_k$ of the dictionaries $\widetilde A_k$ and $\widetilde B_k$ by $\ell_0=2$ and
\[
\ell_k = N_k \ell_{k-1}.
\]
We assume that each word of $\widetilde A_k$ (respectively $\widetilde B_k$) is the concatenation of $N_k$ words of $\widetilde A_{k-1}$ (respectively $\widetilde B_{k-1}$). We define the corresponding concatenated subshifts
\[
\widetilde X^A_k := \langle \widetilde A_k \rangle, \quad \widetilde X^B_k := \langle \widetilde B_k \rangle.
\]
We note $\widetilde L_k := \widetilde A_k \bigsqcup \widetilde B_k$ and  assume that the concatenation of two words of $\widetilde L_k$ is a subword of the concatenation of two words of $\widetilde L_{k+1}$. We define the  corresponding concatenated subshift
\[
\widetilde X_k := \langle \widetilde L_k \rangle.
\]

\begin{lemma} \label{lemma:OnedimensionalConcatenatedSubshift}
Let $\widetilde{\mathcal{A}}$ be a finite alphabet. Let $(N_k)_{k\geq0}$ be a  sequence of integers, $N_k\geq 4$, $(\ell_k)_{k\geq0}$ be a sequence defined inductively by $\ell_0=2$, $\ell_k = N_k \ell_{k-1}$, and $(\widetilde L_k)_{k\geq0}$ be a sequence of dictionaries of size $(\ell_k)_{k\geq0}$ in dimension $1$ over the alphabet $\widetilde{\mathcal{A}}$. We assume that, for every $k\geq0$, every word in $\widetilde L_{k}$ is the concatenation of $N_k$ words of $\widetilde L_{k-1}$, and that  the concatenation of two words of $\widetilde L_k$ is a subword of the concatenation of two words of $\widetilde L_{k+1}$. Let $\widetilde X := \bigcap_{k\geq0} \langle \widetilde L_k \rangle$. Then 
\begin{enumerate}
\item \label{Item:OnedimensionalConcatenatedSubshift_1} $\langle \widetilde L_{k+1} \rangle \subseteq \langle \widetilde L_{k} \rangle$ for every $k \geq 0$,
\item \label{Item:OnedimensionalConcatenatedSubshift_2} $\widetilde X = \Sigma^1(\mathcal{A},\widetilde{\mathcal{F}})$ where $\widetilde{\mathcal{F}} := \bigsqcup_{n\geq0} \widetilde{\mathcal{F}}_n$  and $\widetilde{\mathcal{F}}_n$ is the set of words of length $n$  that are not subwords of any concatenation of two words of $\widetilde L_{k}$ for some $\ell_k \geq n$,
\item \label{Item:OnedimensionalConcatenatedSubshift_3} for every $n\geq0$ and $\ell_k \geq n$, $\mathcal{L}(\widetilde X,n) = \mathcal{L}(\langle \widetilde L_k \rangle, n)$. (In other words, a subword of length $n$ of the concatenation of two words of $\widetilde L_k$ is globally $\widetilde{\mathcal{F}}$-admissible.)
\end{enumerate}
\end{lemma} 

The previous lemma tells us that  $\widetilde X = \bigcap_{k\geq0}\widetilde X_k$  is generated by the set of forbidden words $\widetilde{\mathcal{F}} := \bigsqcup_{n\geq0} \widetilde{\mathcal{F}}_n$ where $\widetilde{\mathcal{F}}_n$ is the set of words of length $n$  that  are not subwords of the concatenation of two words $w$ and $w'$ taken in $\widetilde A_k \sqcup \widetilde B_k$. A direct consequence of this is that the reconstruction function of $\widetilde X$ satisfies
\[
 \ R^{\widetilde X}(n) \leq n \mbox{ for every }n \geq1.
\]
%I think the equality was wrong.
Later on, we will choose a suitable Turing machine $\widetilde{\mathbb{M}}$ which enumerates $\widetilde{\mathcal{F}}$ in such a way that the hypotheses \ref{item:ReconstructionFunctionEstimate_1} and \ref{item:ReconstructionFunctionEstimate_3} of proposition \ref{proposition:ReconstructionFunctionEstimate} are satisfied.

Let $\widetilde{\widetilde X}_k, \widetilde{\widetilde X}_k^A, \widetilde{\widetilde X}^B_k$ be the corresponding vertically aligned subshifts
\begin{gather*}
\widetilde{\widetilde X}_k := \{ \widetilde x \in \Sigma^2(\widetilde{\mathcal{A}}) :  \forall\, v \in \mathbb{Z}, \ (\widetilde x_{(u,v)})_{u \in\mathbb{Z}}) = (\widetilde x_{(u,0)})_{u \in\mathbb{Z}} \in \langle\widetilde L_k \rangle \big\}, \\
\widetilde{\widetilde X}_{k}^A := \{ \widetilde x \in  \widetilde{\widetilde X}_k : \left(\widetilde x_{(u,0)}\right)_{u \in \mathbb{Z}} \in  \langle \widetilde A_k \rangle \}, \quad \widetilde{\widetilde X}_{k}^B := \{\widetilde x \in  \widetilde{\widetilde X}_k : \left(\widetilde x_{(u,0)}\right)_{u \in \mathbb{Z}} \in \langle \widetilde B_k \rangle \}.
\end{gather*}
Then Lemma \ref{lemma:OnedimensionalConcatenatedSubshift} implies that
\[
 \mathcal{L}(\widetilde{\widetilde X},n) = \mathcal{L}(\widetilde{\widetilde X}_k,n) \mbox{ for every } n\geq1 \mbox{ and } \ell_k \geq n.
\]
By the simulation theorem, as $\Pi(X) = \widetilde{\widetilde X}$ ($\Pi$ is defined in Equation (\ref{pi_functions})), we obtain
\[
\Pi(\mathcal{L}(X,n)) = \mathcal{L}(\widetilde{\widetilde X},n) \mbox{ for every } n\geq 1.
\]
\begin{definition}
We denote for $k\geq 0$ by $L_k$, $A_k$ and $B_k$ the dictionaries of size $\ell_k$ given by,
\begin{enumerate}
\item $L_k := \mathcal{L}(X,\ell_k) \subseteq \mathcal{A}^{\llbracket 1,\ell_k \rrbracket^2}$,
\item $A_k := \big\{ p \in L_k : \Pi(p) \in \mathcal{L}(\widetilde{\widetilde X}_k^A,\ell_k) \big\}$,
\item $B_k := \big\{ p \in L_k : \Pi(p) \in \mathcal{L}(\widetilde{\widetilde X}_k^B,\ell_k) \big\}$.
\end{enumerate}
We also denote by $X_k,X_k^A, x_k^B$ the corresponding concatenated subshifts
\[
X_k := \langle L_k \rangle, \ \  X_k^A := \langle A_k \rangle, \ \ X_k^B := \langle B_k \rangle.
\]
\end{definition}

Notice that we obtain a similar structure as in the one-dimensional setting
\[
X = \bigcap_{k\geq0} X_k, \ \ X_{k+1} \subseteq X_k,  \ \  X_k^A \cup X_k^B \subseteq X_k, \ \ X_{k+1}^A \subseteq X_k^A, \ \ X_{k+1}^B \subseteq X_k^B.
\]
Contrary to what happens in the case of $\widetilde X_k^A,\widetilde X_k^B$, the notion of topological entropy will be sufficient to estimate the complexity of the intermediate subshifts $X_k^A$ and $X_k^B$. The entropy will be evaluated using the frequency of the symbol $0$ in the horizontal direction and the duplication of that symbol in the vertical direction. We recall several definitions. See for instance Walters~\cite{Walters1975} and Keller~\cite{Keller1998} for further references.

\begin{definition} \label{definition:EntropyDefinition}
Let $\mathcal{A}$ be a finite set and $X \subseteq \Sigma^d(\mathcal{A})$ be a subshift. 
\begin{enumerate}
\item \label{item:EntropyDefinition_1} The {\it topological entropy} of $X$ is the non negative real number
\[
h_{top}(X) := \lim_{n\to+\infty} \frac{1}{n^d}\log \Card(\mathcal{L}(X,n)).
\]
\item \label{item:EntropyDefinition_2} The {\it canonical generating partition of $\Sigma^d(\mathcal{A})$} is the partition
\begin{equation*}
\label{eq.G}
\mathcal{G}:=\{[a]_0:a\in\mathcal{A}\}.
\end{equation*}
\item \label{item:EntropyDefinition_3} The {\it common refinement} of two partitions $\mathcal{P}$ and $\mathcal{Q}$ of $\Sigma^d(\mathcal{A})$ is the partition
\[ \mathcal{P}\bigvee\mathcal{Q}:=\big\{P\cap Q: P\in\mathcal{P},\, Q\in\mathcal{Q}\, \big\}. \]
\item \label{item:EntropyDefinition_4} The {\it dynamical partition of support $S \subseteq \mathbb{Z}^2$} of a partition $\mathcal{P}$ is the partition
\[ 
\mathcal{P}^S := \bigvee_{u \in S} \sigma^{-u}(\mathcal{P}). 
\]
\item \label{item:EntropyDefinition_5} The {\it entropy of a finite partition} $\mathcal{P}$ with respect to an invariant measure $\mu$ is the quantity
\[ 
H(\mathcal{P},\mu):=\sum_{P\in\mathcal{P}}-\mu(P)\ln(\mu(P)). 
\]
\item \label{item:EntropyDefinition_6} The {\it relative entropy} of a finite partition $\mathcal{P}$ given a partition $\mathcal{Q}$ is the non negative real number
\[ 
H(\mathcal{P}\mid\mathcal{Q},\mu) = \int\!H(\mathcal{P},\mu_{x}^{\mathcal{Q}}) \, d\mu(x),
\]
where $\mu_x^{\mathcal{Q}}$ is the conditional measure with respect to $\mathcal{Q}$.
\item \label{item:EntropyDefinition_7} The Kolmogorov-Sinai entropy of an invariant measure $\mu$ is the quantity
\[
h(\mu) :=\sup\Big\{\lim_{n\to+\infty}\frac{1}{n^d}H\big(\mathcal{P}^{\llbracket 1,n\rrbracket^d},\mu\big):\mathcal{P}\ \text{is a finite partition of $\Sigma^d(\mathcal{A})$}\Big\}.
\]
A measure $\mu$ supported on $X$ satisfying $h(\mu) = h_{top}(X)$ is called a {\it measure of maximal entropy}.
\end{enumerate}
\end{definition}

The following proposition is standard (see the references above).

\begin{proposition}
Let $\mathcal{A}$ be a finite set and $X \subseteq \Sigma^d(\mathcal{A})$ be a subshift. 
\begin{enumerate}
\item There exists an ergodic invariant probability measure $\mu$ supported in $X$ such that
\[
h_{top}(X) = h(\mu).
\]
Such a measure is called {\it measure of maximal entropy}.
\item The {\it Kolmogorov-Sinai entropy} of an invariant probability measure $\mu$ satisfies
\[
h(\mu) = \lim_{n\to+\infty} \frac{1}{n^d} H(\mathcal{G}^{\llbracket 1,n\rrbracket^d},\mu),
\]
where $\mathcal{G}$ is the canonical generating partition of item \eqref{item:EntropyDefinition_2} in definition \ref{definition:EntropyDefinition}.
\end{enumerate}
\end{proposition} 

Let $(\beta_k)_{k\geq0}$ be a sequence of inverse temperatures going to infinity. The heart of our proof is a double estimate of the pressure of $\beta_k \varphi$ that prescribes the statistics of the equilibrium measures. At low temperature an equilibrium measure tends to a minimizing measure that maximizes the topological entropy of the Mather set. As  $\text{\rm Mather}(\varphi) \subseteq X$ and $X$ is obtained as a decreasing sequence of $X_k$, each containing two distinguished subshifts $X_k^A$ and $X_k^B$, we obtain a zero-temperature chaotic behavior by choosing alternatively
\[
\begin{cases}
h_{top}(X_k^A) \ll h_{top}(X_k^B) &\text{for $k$ even,  $ k\to+\infty$}, \\
h_{top}(X_k^B) \ll h_{top}(X_k^A) &\text{for $k$ odd,  $k\to+\infty$}.
\end{cases}
\] 
Where  $a_k  \ll b_k$ for $k$ even,  $ k\to+\infty$ means that $\displaystyle\lim_{k\to+\infty} \frac{a_{2k}}{b_{2k}} = 0$, analogously for $k$ odd.
Because of the duplication process, the topological entropy can be estimated using the frequency of the symbol $0$. Let $f_k^A$ (respectively $f_k^B$)  be the largest frequency of the symbol $0$ in the words of $\widetilde A_k$ (respectively $\widetilde B_k$)
\[
f_k^A := \max_{p \in \widetilde A_k} f_k^A(p), \quad f_k^A(p) := \frac{1}{\ell_k} \Card \big(\{ i\in \llbracket 1, \ell_k \rrbracket :  p(i) = 0 \} \big).
\]
Our construction of $X_k^{A}$ (and $X_k^{B}$ ) will satisfy that
\[
\begin{cases}
f_k^A \ll f_k^B \ll 1&\text{for $k$ even,  $ k\to+\infty$}, \\
f_k^B \ll f_k^A \ll 1 &\text{for $k$ odd,  $ k\to+\infty$}.
\end{cases}
\] 

We now explain the double estimate for the pressure that are at the heart of the proof: item 3 of Lemma \ref{lemma:BoundFromBellow} and item 2 of Lemma \ref{lemma:EntropyEstimate_2}.

  The first estimate is standard
\[
P(\beta_k \varphi) = \sup_\mu \{ h(\mu) -\beta_k \mu([\mathcal{F}]) \geq h(\mu_k^B) -\beta_k \mu_k^B([\mathcal{F}]),
\]
where $\mu_k^B$ is an ergodic maximal entropy measure of the subshift $X_k^B$ and $k$ is even.

Using the fact that a  configuration in the support of $\mu_k^B$ is a tilling of square patterns of size $\ell_k$ that are globally $\mathcal{F}$-admissible, we obtain easily the following estimates.

\begin{lemma} \label{lemma:BoundFromBellow}
Let $k \geq0$.
\begin{enumerate}
\item \label{item:BoundFromBellow_1} For every ergodic probability measure $\mu$  satisfying $\Supp(\mu) \subseteq X_k^B$
\[
\mu([\mathcal{F}]) \leq \frac{2D}{\ell_k}.
\]
\item \label{item:BoundFromBellow_2} The topological entropy of the shift $X_k^B$ is bounded from below by 
\[
h_{top}\big(X_k^B\big) \geq \ln(2) f_k^B.
\]
\item  \label{item:BoundFromBellow_3} The pressure of $\beta_k \varphi$ is bounded from below by 
\[
P(\beta_k \varphi) \geq \ln(2) f_k^B  -2D \frac{\beta_k}{\ell_k}.
\]
\end{enumerate}
\end{lemma}
A similar estimate is also valid for $X_k^A,f_k^A$, instead of $X_k^B,f_k^B$. The two parameters $\beta_k$ and $\ell_k$ will be chosen so that the following first constraint is valid
\begin{equation}
\begin{cases}
\frac{\beta_k}{\ell_k} \ll f_k^B, \  \text{for $k$ even,  $ k\to+\infty$},\\
\frac{\beta_k}{\ell_k} \ll f_k^A, \  \text{for $k$ odd,  $ k\to+\infty$}.
\end{cases}\tag{C1} \label{equation:Constraint_1}
\end{equation}

The second estimate is a bound from above on the pressure at $\beta_k\varphi$. In order to obtain it, we need to introduce a sequence of intermediate scales $(\ell'_k)_{k\geq0}$, two intermediate dictionaries $\widetilde A'_k$ and $\widetilde B'_k$, and assume that $\widetilde A_k$ and $\widetilde B_k$ are built over $\widetilde A'_k$ and $\widetilde B'_k$ in the following way.

\begin{definition} \label{Definition:IntermediateScales}
For every $k\geq0$ we define
\[
\ell'_k = N'_k \ell_{k-1},
\]
where $N'_k \geq 2$ is an integer and $N_k$ is a multiple of $N'_k$ with $N_k/N'_k\geq2$. Let $\widetilde A_k'$ (respectively $\widetilde B'_k$) be some intermediate dictionary over the alphabet $\widetilde{\mathcal{A}}$ of size $\ell'_k$
\[
\widetilde A_k', \widetilde B'_k \subseteq \widetilde{\mathcal{A}}^{\llbracket 1,\ell_k'\rrbracket},
\] 
satisfying the property that every word of $\widetilde A'_k$  (respectively $\widetilde B'_k$) is obtained by concatenating $N'_k$ words of $\widetilde A_{k-1}$ (respectively $N'_k$ words of $\widetilde B_{k-1}$). We define
\begin{equation}\label{equation_2}
    \widetilde L'_k := \widetilde A'_k \bigsqcup \widetilde B'_k.
\end{equation}

Assume also that each word of $\widetilde A_k$ (respectively $\widetilde B_k$) is a concatenation of $N_k/N'_k$ words of $\widetilde A'_k$ (respectively $\widetilde B'_k$). 
\end{definition}

We introduce the following notations.

\begin{enumerate}
\item $R'_k := 2 R^{\widehat X}(\ell'_k)+1$  be {\it the reconstruction length at the scale $\ell'_k$} (see Definition~\ref{definition:ReconstructionFunction} for the definition  of the reconstruction function $R^{\widehat X}$), 
\item $M'_k \subseteq \mathcal{A}^{\llbracket 1,R'_k \rrbracket^2}$ be the set of patterns of size $R'_k$ that are locally $\mathcal{F}$-admissible
\[
M'_k := \big\{ w\in\mathcal{A}^{\llbracket 1,R'_k\rrbracket^2}:\forall\,p \in \mathcal{F},\ \forall\,u\in\llbracket 0,R'_k-D\rrbracket^2,\ p \not\sqsubset\sigma^u(w) \big\},
\]
(the set $M'_k$ is called {\it the  reconstruction cylinder at scale $\ell'_k$}),
\item $T'_k := \big(\big\lfloor\frac{R'_k}{2}\big\rfloor -\ell'_k,\big\lfloor\frac{R'_k}{2}\big\rfloor -\ell'_k \big) \in \mathbb{Z}^2$ be a translation vector to the center of ${\llbracket 1,R'_k\rrbracket^2}$,
\item $Q'_k := T'_k+ \llbracket 1, 2\ell'_k \rrbracket^2 \subseteq \mathbb{Z}^2$ be the central block of indices for which for every $w \in M'_k$, $w|_{Q'_k}$ is globally $\mathcal{F}$-admissible.
\end{enumerate}

The following lemma shows that an ergodic equilibrium measure $\mu_{\beta_k}$ for $\beta_k\varphi$ at low temperature ($\beta_k \gg 1$) has the tendency to give a large mass to  sets of configurations that minimize $\varphi$, that is, to sets of configurations that are locally $\mathcal {F}$-admissible. Using the trivial estimate
\[
0 \leq P(\beta_k\varphi) = h(\mu_{\beta_k}) - \int\! \beta_k \varphi \, d\mu_{\beta_k} \leq h_{top}(\Sigma^2(\mathcal{A}) )- \beta_k \mu_{\beta_k}([\mathcal{F}]),
\]
one proves easily the following bound.

\begin{lemma} \label{lemma:sizeNeighborhoodSubshift}
For every $k$ and every ergodic equilibrium measure $\mu_{\beta_k}$,
\begin{equation*}
\label{eq.epsilon_k}
\dis \mu_{\beta_k}\left(\Sigma^2(\mathcal{A}) \setminus [M_k']\right) \leq \frac{(R_k')^2}{\beta_k}\ln( \Card(\mathcal{A})) =: \epsilon_k.
\end{equation*}
\end{lemma}
The parameter $\epsilon_k$ is supposed to be small at low temperature.  Also,  $\epsilon_k$  must actually be negligible compared to the frequency $f_{k-1}^B$. More precisely
\begin{equation*}
\epsilon_k \ll f_{k-1}^B,  \\ \  \text{for $k$ even,  $ k\to+\infty$},
\end{equation*} 
and using $ H(\epsilon_k) := -\epsilon_k \log(\epsilon_k) -(1-\epsilon_k) \log(1-\epsilon_k) $, we also need  

\begin{equation*}
H(\epsilon_k)  \ll f_{k-1}^B  \\ \  \text{for $k$ even,  $ k\to+\infty$}.
\end{equation*} 
We simplify these two constraints by using $-\epsilon_k \log(\epsilon_k) \ll \sqrt{\epsilon_k}$, and thus, by imposing
\begin{equation}
\begin{cases}
\frac{(R'_k)^2}{\beta_k} \ll (f_{k-1}^B)^2, \  \text{for $k$ even,  $ k\to+\infty$}, \\
\frac{(R'_k)^2}{\beta_k} \ll (f_{k-1}^A)^2, \  \text{for $k$ odd,  $ k\to+\infty$}.
\end{cases}\tag{C2} \label{equation:Constraint_2}
\end{equation}

A typical configuration for $\mu_{\beta_k}$ sees the reconstruction cylinder with probability $1-\epsilon_k$.  

\begin{definition} Consider the space $\Sigma^2(\mathcal{A})$. We define
\begin{enumerate}
\item  The {\it canonical base partition} is given by
\begin{gather*}
{\widetilde{\mathcal{G}}} := \{{\widetilde G_0},{\widetilde G_1}, {\widetilde G_2} \}, \ \
\widetilde G_{\widetilde a}  := \{ x \in \Sigma^2(\mathcal{A}) : \pi(x(0)) = \widetilde a \},  \ \forall\, \widetilde a \in \widetilde{\mathcal{A}}.
\end{gather*}
\item The {\it reconstruction partition at scale $\ell'_k$} is the partition $\mathcal{U}_k$ given by
\[
\mathcal{U}_k := \{ [M'_k], \Sigma^2(\mathcal{A}) \setminus [M'_k] \}.
\]
\end{enumerate}
\end{definition}
Notice that $\widetilde{\mathcal{G}}$ is a partition of $\Sigma^2(\mathcal{A})$ and not of $\Sigma^1(\widetilde{\mathcal{A}})$. The symbols coming from the simulation theorem are hidden. The only symbols that remain visible are those from the one-dimensional subshift.

An upper bound on the pressure of $\beta_k \varphi$ is given by the entropy of the equilibrium measure 
\[
P(\beta_k \varphi) \leq  h(\mu_{\beta_k}).
\]

We decompose the computation of the entropy of $\mu_{\beta_k}$ into 3 terms using a standard identity on relative entropies
\begin{gather*}
h(\mu_{\beta_k}) = h_{rel}(\mu_{\beta_k}) + \limsup_{n\to+\infty} \frac{1}{n^2} \Big[ H \Big( {\widetilde{\mathcal{G}}}^{\llbracket 1,n \rrbracket^2} \mid \mathcal{U}_k^{\llbracket 0,n-R_k' \rrbracket^2}, \mu_{\beta_k} \Big)  + H \Big( \mathcal{U}_k^{\llbracket 0,n-R_k' \rrbracket^2}, \mu_{\beta_k} \Big) \Big].
\end{gather*}
The first term $h_{rel}(\mu_{\beta_k})$ is the {\it relative entropy of $\mu_{\beta_k}$ at scale $\ell'_k$} 
\begin{gather*}
 h_{rel}(\mu_{\beta_k}):=\lim_{n\to+\infty}\frac{1}{n^2}H\Big(\,\mathcal{G}^{\llbracket 1,n\rrbracket^2}\mid {\widetilde{\mathcal{G}}}^{\llbracket 1,n\rrbracket^2} \bigvee \mathcal{U}_k^{\llbracket 0,n-R'_k \rrbracket^2},\mu_{\beta_k}\,\Big).
\end{gather*}
The term $h_{rel}(\mu_{\beta_k})$ is dominant; it computes the entropy of the canonical generating partition $\mathcal{G}$ in $\Sigma^2(\mathcal{A})$, see Definition \ref{definition:EntropyDefinition}, knowing the fact that the $\widetilde{\mathcal{A}}$-symbols are fixed and that large patterns in $\mathcal{A}^{\llbracket 1,n \rrbracket^2}$ are tiled by almost non overlapping locally admissible patterns of size $R'_k$.  The second term computes the entropy of the canonical base partition knowing the fact that most the time a configuration is vertically aligned. That term is negligible. The last term computes the entropy of a two set partition where one of the sets, the reconstruction cylinder, has large measure $\mu_{\beta_k}([M'_k]) > 1-\epsilon_k$. That term is again negligible.

We obtain easily the following estimates.

\begin{lemma} \label{lemma:EntropyEstimate_1}
For every $k$ and every equilibrium measure $\mu_{\beta_k}$,
\begin{enumerate}
\item $\displaystyle \limsup_{n\to+\infty} \frac{1}{n^2}H \big( \mathcal{U}_k^{\llbracket 0,n-R_k' \rrbracket^2}, \mu_{\beta_k} \big) \leq H(\epsilon_k)$,
\item $\displaystyle \limsup_{n\to+\infty} \frac{1}{n^2} H\left({\widetilde{\mathcal{G}}}^{\llbracket 1,n\rrbracket^2}\mid\mathcal{U}_k^{\llbracket 0,n-R'_k\rrbracket^2},\mu_{\beta_k}\right) \leq \left(\frac{8}{R_k'}+ \epsilon_k\right) \ln(\Card(\widetilde{\mathcal{A}}))$.
\end{enumerate}
\end{lemma}

The hardest part of the proof is to bound from above the relative entropy of $\mu_{\beta_k}$ at scale $\ell'_k$ with respect to the complexity of 1-dimensional globally $\widetilde{\mathcal{F}}$-admissible words of length $\ell'_k$ (a square tile of size $R'_k$ gives in its center a square tile  of size $\ell'_k$ that is globally admissible). Item \eqref{item:EntropyEstimate_2_1} in the following lemma shows that the frequency of the symbol $0 (0^{'}\wedge0^{''})$ is dominated by the one of the symbol $2$ if $k$ is even. Item \eqref{item:EntropyEstimate_2_2} is the most difficult estimate to prove. The computation depends on a particular choice of the language $\widetilde L^{'}_k$ (see equation \ref{equation_2} ) with respect to $\widetilde L_{k-1}$. If $k$ is even, the frequency of the symbol $0$ in words in $\widetilde B'_k$ coincides with the frequency $f_{k-1}^B$, the frequency of $0$ in $\widetilde A'_k$ is negligible (of the form $f^A_{k-1}/N'_k$). If $\mu_{\beta_k}$ gives some positive mass to ${\widetilde G_1}$, then the proportion of the space of configurations that can be covered by words in $B'_k$ (words containing only the symbols $0$ and $2$) is thus less than $\mu_{\beta_k} \big(\Sigma^2(\mathcal{A}) \setminus {\widetilde G_1} \big)$. In particular Item \eqref{item:EntropyEstimate_2_3} gives us the means to show that the support of the measure $\mu_{\beta_k}$ is in ${\widetilde G_2}$ for $k$ even and in ${\widetilde G_1}$ for $k$ odd.

\begin{lemma} \label{lemma:EntropyEstimate_2}
For every $k$ and every equilibrium measure $\mu_{\beta_k}$,
\begin{enumerate}
\item \label{item:EntropyEstimate_2_1} $\displaystyle \mu_{\beta_k}(\widetilde G_0) \leq  \frac{2}{N'_k}  f_{k-1}^A + (1-N_{k-1}^{-1})^{-1} f^B_{k-1} + \epsilon_k$,
\item \label{item:EntropyEstimate_2_2} if $k$ is even, then
\begin{multline*}
h_{rel}(\mu_{\beta_k})  \leq \Big( \frac{2}{N'_k}  f_{k-1}^A + \big(1-N_{k-1}^{-1} \big)^{-1} \big(\mu_{\beta_k} \big(\Sigma^2(\mathcal{A}) \setminus {\widetilde G_1} \big) + \epsilon_k \big) f^B_{k-1}\Big) \ln(2) \\
+ \frac{1}{\ell'_k} \ln(\Card(\widetilde{\mathcal{A}})) +  \frac{\ln(C_k')}{\ell'_k{}^2} + \epsilon_k \ln (2\Card(\widehat{\mathcal{A}})),
\end{multline*}
\quad$\displaystyle P(\beta_k \varphi) \leq h_{rel}(\mu_{\beta_k}) + \left(\frac{8}{R_k'}+ \epsilon_k\right) \ln(\Card(\widetilde{\mathcal{A}})) + H(\epsilon_k)$.
\item \label{item:EntropyEstimate_2_3} if $k$ is odd, then the previous estimate is valid with $f_{k-1}^A$ and $f_{k-1}^B$ permuted and $ {\widetilde G_1}$ replaced by $ {\widetilde G_2}$.
\end{enumerate}
\end{lemma}
The term $f_{k-1}^B\ln(2)$ in item \eqref{item:EntropyEstimate_2_2} is the entropy of duplicated and vertically aligned words taken in an intermediate dictionary $B'_k$ of scale $\ell'_k$. The term $f_{k-1}^A\ln(2)$ is interpreted similarly. The term $f^A_{k-1}/(f^B_{k-1}N'_k)$ reflects the fact the ratio of the number of $0$ between words in $A'_k$ and $B'_k$ is $1/N'_k$ for $k$ even (a word in $B'_k$ contains much more $0$ that  a word in $A'_k$). The term $\ln(C'_k)/\ell'_k{^2}$ converges to the entropy of the simulating SFT. Though the entropy of the Aubrun-Sablik subshift has zero entropy, it is not enough to conclude. This argument seems to be missing in the proof in \cite{ChazottesHochman2010}. The purpose of Proposition \ref{proposition:ComplexityFunctionEstimate} is to give a stronger a priori bound of the growth of the relative complexity function  of the simulating SFT provided the set of forbidden patterns $\widetilde{\mathcal{F}}$ are enumerated in a special way. Then, we will use the estimate 
\begin{gather*}
\frac{1}{\ell'_k} \ll f_{k-1}^B,  \  \text{for $k$ even,  $ k\to+\infty$}.
\end{gather*}
In order for $f_{k-1}^B \ln(2)$ to be the dominant term in item (\ref{item:BoundFromBellow_3}) of Lemma~\ref{lemma:BoundFromBellow} and item (\ref{item:EntropyEstimate_2_2}) of Lemma~\ref{lemma:EntropyEstimate_2}, we assume that $\ell'_k$ and $\ell_k$ have been chosen according to the following additional constraint: we assume
\begin{gather}
\begin{cases}
\frac{f_{k-1}^A}{N'_k} \ll f_{k-1}^B,  \  \text{for $k$ even,  $ k\to+\infty$},\\
\frac{f_{k-1}^B}{N'_k} \ll f_{k-1}^A,  \  \text{for $k$ odd,  $ k\to+\infty$}.
\end{cases} \tag{C3} \label{equation:Constraint_3}
\end{gather}

Notice that the two conditions \eqref{equation:Constraint_1} and\eqref{equation:Constraint_2} give us an interval of temperatures as follows:
\[
\ell_{k}f_{k-1}^B \gg \beta_k \gg \frac{(R'_k)^2}{(f_{k-1}^B)^2} ,  \  \text{for $k$ even,  $ k\to+\infty$}. 
\]
These two constraints imply an upper bound of $N'_k$ with respect to $N_k$.  Recalling that $\ell'_k = N'_k \ell_{k-1}$, $R'_k \geq \ell'_k$  and $\ell_k =N_k \ell_{k-1}$, we have
\[
(N'_k)^2 = \left(\frac{\ell'_k}{\ell_{k-1}}\right)^2 \leq \left(\frac{R'_k}{\ell_{k-1}}\right)^2\ll \ell_k\frac{(f^B_{k-1})^3}{\ell^2_{k-1}} = N_k\frac{(f^B_{k-1})^3}{\ell_{k-1}} \ll N_k,  \  \text{for $k$ even,  $ k\to+\infty$}, 
\]
which implies
\[
  N'_k \ll N_k,  \  \text{for $k$ even,  $ k\to+\infty$}.
\]
On the other hand, condition \eqref{equation:Constraint_3} implies a strong constraint on the lower bound of $N'_k$ with respect to $N_k$ as follows
\[
\frac{f_{k-1}^A}{N'_k} \ll f_{k-1}^B \ll f_{k-1}^A,  \  \text{for $k$ even,  $ k\to+\infty$},
\]
which implies
\[
1 \ll N'_k,  \  \text{for $k$ even,  $ k\to+\infty$}.
\]

As a conclusion, we are forced to choose $N'_k$ satisfying 
\[
1 \ll N'_k \ll N_k,  \  \text{for $k$ even,  $ k\to+\infty$}.
\]

The intermediate scale $\ell'_k$ is fundamental for our proof, since it allows us to choose the sequence $(\beta_k)_k$ used to prove the existence of the chaotic behavior.  This part of the argument is not given explicitly in the literature and proofs given by other authors.

We use Proposition \ref{proposition:ReconstructionFunctionEstimate} to obtain an a priori upper bound of the growth of the reconstruction function. This bound is one of the main ingredients of the construction and,  it seems not be highlighted in the other papers about the question.The shape of the reconstruction function has  logarithmic growth. We don't need the exact growth but an explicit growth to obtain a recursive sequence; see next Definition \ref{definition:RecursiveSequence}. The upper bound depends on two properties of the  time enumeration function (Definition \ref{Definition:TimeEnumerationFunction}) of the 1-dimensional set of forbidden patterns. More precisely, the time enumeration function must satisfy that the forbidden words are enumerated successively according to their length (and thus the function is non-decreasing), and that the time to enumerate all words of length $n$ is at most polynomial in $n$. 

We construct by induction $\widetilde{\mathcal{F}_n}$, the full set of forbidden words of $\widetilde X$ of length $n$. We define a primary sequence of scales $(\ell_k)_{k\geq0}$ and a intermediate  sequence of scales $(\ell'_k)_{k\geq0}$ so that, by choosing first $N'_k$ large enough and $\ell'_k = N'_k \ell_{k-1}$, \eqref{equation:Constraint_3} is satisfied, by choosing secondly $N_k$ large enough and $\ell_k = N_k \ell_{k-1}$, \eqref{equation:Constraint_1} and \eqref{equation:Constraint_2} are satisfied and $\beta_k$ is chosen. Essentially it all comes down to check that 
\[
\ell_{k} (f_{k-1}^B)^3 \gg  (R'_k)^2,  \  \text{for $k$ even,  $ k\to+\infty$},
\]
\[
\ell_{k} (f_{k-1}^A)^3 \gg  (R'_k)^2  \  \text{for $k$ odd,  $ k\to+\infty$}.
\]
As $\widetilde L_k$ is constructed by concatenating $N_k$ words of $\widetilde L_{k-1}$, it is clear that a Turing machine might be described such that its time enumeration function is at most exponential independently of the choice of $N_k$. Proposition \ref{proposition:ReconstructionFunctionEstimate} shows that it is enough to choose $N_k$ so that $\ell_k (f_{k-1}^A)^3$ is super-exponential in $\ell'_k$. 

Finally, let us assume that the following constraint holds
\begin{gather}
\begin{cases}
f_{k}^B = f_{k-1}^B &\text{if $k$ is even}, \\
f_k^A = f_{k-1}^A &\text{if $k$ is odd}.
\end{cases} \tag{C4} \label{equation:Constraint_4}
\end{gather}
Then the double estimates, given in  Lemma \ref{lemma:BoundFromBellow} item 3 and  Lemma \ref{lemma:EntropyEstimate_2} item 2,  can be reduced to the estimate (for $k$ even)
\begin{gather*}
\ln(2) f_{k-1}^B + o( f_{k-1}^B) \leq P(\beta_k \varphi) \leq  \mu_{\beta_k} \big(\Sigma^2(\mathcal{A}) \setminus {\widetilde G_1} \big) \ln(2) f_{k-1}^B + o( f_{k-1}^B),
\end{gather*}
using the analogous inequalities, which holds for $k$ odd, $A$ and $\widetilde G_2$, instead of $B$ and $\widetilde G_1$, we obtain
\begin{gather*}
\lim_{k\to+\infty}\mu_{\beta_{2k}}({\widetilde G_1}) = 0, \quad \lim_{k\to+\infty}\mu_{\beta_{2k+1}}({\widetilde G_2}) = 0.
\end{gather*}
We observe that the two a priori estimates in \ref{proposition:ComplexityFunctionEstimate} and \ref{proposition:ReconstructionFunctionEstimate} could be simplified drastically. The only property we need is to have an a priori explicit bound (exponential, super-exponential, or more) of the growth of the reconstruction function, and an a priori sub-exponential bound of the growth of the relative complexity function.

We conclude this section by giving the complete proof Theorem \ref{main.theorem.generalized} assuming the  estimates in Lemmas \ref{lemma:BoundFromBellow}, \ref{lemma:sizeNeighborhoodSubshift}, \ref{lemma:EntropyEstimate_1}, \ref{lemma:EntropyEstimate_2}, and assuming the constraints \eqref{equation:Constraint_1}--\eqref{equation:Constraint_4}.

\begin{proof}[Proof of Theorem \ref{main.theorem.generalized}]
Let $\mu_{\beta_k}$ be an equilibrium measure at inverse temperature $\beta_k$. Assume $k$ is an even number.  Let $\mu_k^B$ be the measure of maximal entropy of the concatenated  subshift $X_k^B$. On the one hand, from  Lemma \ref{lemma:BoundFromBellow}, we have that
\[ P(\beta_k \varphi) \geq h(\mu_k^B) -\int \! \beta_k \varphi \, d\mu_k^B \geq \ln(2) f_k^B  -2D \frac{\beta_k}{\ell_k}. \]
From the constraints \eqref{equation:Constraint_1} and \eqref{equation:Constraint_4}, we have that
\[
\frac{\beta_k}{\ell_k}  \ll f_{k-1}^B = f_k^B, \  \text{for $k$ even,  $ k\to+\infty$, and}
\]
\[
P(\beta_k \varphi) \geq \ln(2) f_{k-1}^B + o(f_{k-1}^B).
\]
On the other hand, by Lemma  \ref{lemma:EntropyEstimate_2}
\begin{align*}
h(\mu_{\beta_k}) &= \lim_{n\to+\infty} \frac{1}{n^2} H\Big( \mathcal{G}^{\llbracket 1,n \rrbracket^2}, \mu_{\beta_k} \Big) \\
&= h_{rel}(\mu_{\beta_k}) + \limsup_{n\to+\infty} \frac{1}{n^2} \Big[ H \Big( {\widetilde{\mathcal{G}}}^{\llbracket 1,n \rrbracket^2} \mid \mathcal{U}_k^{\llbracket 0,n-R_k' \rrbracket^2}, \mu_{\beta_k} \Big)  + H \Big( \mathcal{U}_k^{\llbracket 0,n-R_k' \rrbracket^2}, \mu_{\beta_k} \Big) \Big], \\
P(\beta_k \varphi) &\leq  
\Big( \frac{2}{N'_k}  f_{k-1}^A + (1-N_{k-1}^{-1})^{-1} \Big( \mu_{\beta_k} \big(\Sigma^2(\mathcal{A}) \setminus {\widetilde G_1} \big)  + \epsilon_k \Big) f^B_{k-1}\Big) \ln(2) \\
&\quad  + \frac{1}{\ell'_k} \ln(\Card(\widetilde{\mathcal{A}}))  +  \frac{1}{\ell'_k{}^2}\ln(C'_k) + \epsilon_k \ln (2\Card(\widehat{\mathcal{A}})) \\
&\quad  + \Big( \frac{8}{R'_k}+ \epsilon_k  \Big) \ln(\Card(\widetilde{\mathcal{A}}))+  H(\epsilon_k).
\end{align*}
Constraints \eqref{equation:Constraint_2} and  \eqref{equation:Constraint_3} imply
\[
\epsilon_k \ll f_{k-1}^B \ \ \text{and} \ \ H(\epsilon_k) \ll f_{k-1}^B, \quad \frac{f_{k-1}^A}{N'_k} \ll f_{k-1}^B, \  \text{for $k$ even,  $ k\to+\infty$.}
\]
The fact that $N'_k \to +\infty$ and $\ell_{k-1}f_{k-1}^B \geq1$ implies
\[
\frac{1}{R'_k} \leq  \frac{1}{\ell'_k} \leq \frac{f_{k-1}^B}{N'_k} \ll f_{k-1}^B, \  \text{for $k$ even,  $ k\to+\infty$.}
\]
Proposition \ref{proposition:ComplexityFunctionEstimate} implies
\[
\limsup_{k\to+\infty} \frac{\ln(C'_k)}{\ell'_k} =0 \ \Rightarrow \ \frac{\ln(C'_k)}{\ell'_k{}^2} \ll \frac{1}{\ell'_k}= \frac{1}{N'_k \ell_{k-1}} \ll f_{k-1}^B, \  \text{for $k$ even,  $ k\to+\infty$.}
\]
We finally obtain
\[
\ln(2) f_{k-1}^B + o(f_{k-1}^B) \leq P(\beta_k \varphi) \leq  \mu_{\beta_k} \big(\Sigma^2(\mathcal{A}) \setminus {\widetilde G_1} \big) \ln(2)  f_{k-1}^B + o(f_{k-1}^B), \  \text{for $k$ even, and}
\]
\[
\lim_{k\to +\infty} \mu_{\beta_{2k}}({\widetilde G_1}) = 0.
\]
Finally, using item~\ref{item:EntropyEstimate_2_1} of Lemma~\ref{lemma:EntropyEstimate_2},  $\lim_{k\to+\infty} \mu_{\beta_k}(\widetilde G_0) = 0$, we obtain
\[
\lim_{k\to +\infty} \mu_{\beta_{2k}}({\widetilde G_2}) = 1.
\]
An analogous argument based on item~\ref{item:EntropyEstimate_2_3} of Lemma~\ref{lemma:EntropyEstimate_2} yields \[ \lim_{k\to +\infty} \mu_{\beta_{2k+1}}({\widetilde G_1})  = 1. \]
This concludes the proof.
\end{proof}

\section{The detailed construction}
\label{section.main.construction}

We complete section~\ref{section:OutlineProof} by providing complete proofs of all the previous statements.

\subsection{The one-dimensional effectively closed subshift}

We start by defining an iteration process that will generate the language of $\widetilde X$ over the alphabet $\widetilde{\mathcal{A}}=\{0,1,2\}$. Recall that we use the symbol ``tilde $\sim$'' in all the one-dimensional elements. In a first stage, we will define the values $(\ell_k)_{k\geq0}$ recursively and we define the values of the sequence of inverse temperatures $(\beta_k)_{k\geq0}$. In order to work only with integers, instead of the frequencies $f_k^A,f_k^B$, we shall define the maximum of the number of symbols $0$ counted over all words in $\widetilde A_k$ and $\widetilde B_k$ respectively:
\[
\rho_k^A := \ell_k f_k^A, \quad \rho_k^B := \ell_k f_k^B.
\]

\begin{definition}[The recursive sequence]
\label{definition:RecursiveSequence}\ \\
There exists a partial recursive function $S : \mathbb{N}^4 \to \mathbb{N}^4$
\[
(\ell_k,\beta_k,\rho_k^A,\rho_k^B) = S(\ell_{k-1},\beta_{k-1},\rho_{k-1}^A, \rho_{k-1}^B).
\] 
satisfying $\ell_0=2$, $\beta_0=0$, $\rho_0^A = \rho_0^B = 1$ and defined such that the following holds. In the case $k$ is even:
\begin{enumerate}

\item \label{item:RecursiveSequence_1} $\displaystyle N'_k :=  \Big\lceil \frac{2k \rho_{k-1}^A}{\rho_{k-1}^B} \Big\rceil, \ \ell'_k = N'_k \ell_{k-1}$,
\item \label{item:RecursiveSequence_2} $\displaystyle \beta_k := \Big\lceil\frac{\ell_{k-1}^2 2^{k \ell'_k}}{(\rho_{k-1}^B)^2}\Big\rceil$,
\item \label{item:RecursiveSequence_3} $\displaystyle N_k := N'_k \Big\lceil \frac{k\beta_k}{N'_k \rho_{k-1}^B} \Big\rceil, \ \ell_k = N_k \ell_{k-1}$,
\item \label{item:RecursiveSequence_4} $\displaystyle \rho^A_k =2\rho_{k-1}^A,  \   \rho^B_k =N_k  \rho^B_{k-1}$,
\end{enumerate}
In the case $k$ is odd: $(\ell_k,\beta_k,\rho_k^A,\rho_k^B)$ are computed as before with $A$ and $B$ permuted.
\end{definition}

The previous sequence $(\ell_k,\beta_k,\rho_k^A,\rho_k^B) _{k\geq0}$ has been chosen so that, first the induction step is explicit in terms of simple (computable) operations, and secondly, such that the four constraints (C1)--(C4) are satisfied. We first observe the following inequalities. 

\begin{remark} \label{Remark:InductionChecking}
For all $k\geq1$ we have the following properties:
\begin{enumerate}
\item \label{Item:InductionChecking_1} $2k  \leq N'_k \leq  2k \ell_{k-1} $,
\item $2^{k\ell'_k} \leq  \beta_k \leq \frac{\ell_k}{k}$, 
\item $N_{k-1} \leq N'_k \leq N_k$,
\item if $k$ is odd, $\rho_k^A \geq \rho_k^B$, if $k$ is even, $\rho_k^B \geq \rho_k^A$, 
\item $\displaystyle \beta_k  \leq \frac{\ell_k \beta_{k+1} }{k 2^{(k+1)\ell_k}}\leq \beta_{k+1}$,
\item $f_k^A \ll 1$, $f_k^B \ll 1$, $\displaystyle \frac{N_k}{N'_k} \gg 1$, $\displaystyle \frac{N'_k}{N_{k-1}} \gg 1$, $\displaystyle \frac{\beta_{k+1}}{\beta_k} \gg 1$, 
\item if $k$ is odd, $f_k^A \gg f_k^B$, if $k$ is even, $f_k^A \ll f_k^B$,  when $k\to+\infty$.
\end{enumerate}
\end{remark}

\begin{lemma}
The four constraints (C1)--(C4) are satisfied
\end{lemma}

\begin{proof}
We assume that $k$ is even. The constraint \eqref{equation:Constraint_4} is satisfied thanks to item \ref{item:RecursiveSequence_4} of \ref{definition:RecursiveSequence}. The constraint \eqref{equation:Constraint_3} is satisfied  thanks to item \ref{item:RecursiveSequence_1}  of \ref{definition:RecursiveSequence} and
\[
\frac{f_{k-1}^A}{f_{k-1}^B} \leq \frac{N'_k}{2k} \ll N'_k,  \  \text{for $k$ even,  $ k\to+\infty$.}
\]
The constraint \eqref{equation:Constraint_2} is satisfied thanks to item \ref{item:RecursiveSequence_2} of \ref{definition:RecursiveSequence} (especially the fact that we chose a superexponential growth $2^{k\ell'_k}$ instead of $2^{\ell'_k}$) and the assumption on the bound for the reconstruction function (see Proposition \ref{proposition:ReconstructionFunctionEstimate}),
\[
\limsup_{k\to+\infty}\frac{\ln(R'_k)}{\ell'_k} <+ \infty \ \text{and} \ \beta_k \geq \frac{2^{k\ell'_k}}{(f_{k-1}^B)^2} \gg \frac{(R'_k)2}{(f_{k-1}^B)},  \  \text{for $k$ even,  $ k\to+\infty$.}
\]
The constraint \eqref{equation:Constraint_1} is satisfied thanks to item \ref{item:RecursiveSequence_3} of \ref{definition:RecursiveSequence} and
\[
\beta_k \leq \frac{N_k \rho_{k-1}^B}{k} \leq \frac{\ell_k f_{k-1}^B}{k} \ll \ell_k f_{k-1}^B = \ell_k f_k^B,  \  \text{for $k$ even,  $ k\to+\infty$.}
\]
\end{proof}

We now construct the effectively closed 1-dimensional subshifts $\widetilde A_k$ and $\widetilde B_k$

\begin{definition}
\label{notation:BaseLanguageBis}
For each $k\geq 0$ the dictionaries $\widetilde{A}_k$ and $\widetilde{B}_k$ are made of two words of length $\ell_k$
\[ 
\widetilde{A}_k=\{a_k,1^{\ell_k}\} \subset \widetilde{\mathcal{A}}_1^{\llbracket 1,\ell_k \rrbracket}, \quad \widetilde{B}_k=\{b_k,2^{\ell_k}\} \subset \widetilde{\mathcal{A}}_2^{\llbracket 1,\ell_k \rrbracket}, 
\]
defined by induction as follows:
\begin{enumerate}
\item $\ell_0=2$, $a_0=01$ and $b_0=02$,
\[ \widetilde{A}_0=\{01,11\} \quad \mbox{and} \quad \widetilde{B}_0=\{02,22\}. \]
\item  if $k\geq 1$ is odd 
\begin{equation}
a_k=\underbrace{a_{k-1}a_{k-1}\cdots a_{k-1}}_{N_k\mbox{-times}}, \quad b_k=b_{k-1}2^{(N_k-2)\ell_{k-1}}b_{k-1}, \tag{R1} \label{rule.odd}
\end{equation}
\item if $k\geq 2$ is even
\begin{equation}
a_k=a_{k-1}1^{(N_k-2)\ell_{k-1}}a_{k-1}, \quad
b_k=\underbrace{b_{k-1}b_{k-1}\cdots b_{k-1}}_{N_k\mbox{-times}}.
\tag{R2} \label{rule.even}
\end{equation}
\end{enumerate}
\end{definition}

Notice that by definition, every word in $\widetilde L_k = \widetilde A_k \bigsqcup \widetilde B_k$ is the concatenation of $N_k$ words of $\widetilde L_{k-1}$, and that the concatenation of two words of $\widetilde L_k$ is a word of $\widetilde L_{k+1}$. It follows that the assumptions of Lemma~\ref{lemma:OnedimensionalConcatenatedSubshift} are satisfied. We~now proceed to prove that lemma. Recall that
\[
\widetilde X := \bigcap_{k\geq0}\langle \widetilde L_k \rangle,
\]
and let $\widetilde{\mathcal{F}}:=\bigsqcup_{n\in\NN}\widetilde{\mathcal{F}}(n)$ be the set of all forbidden patterns which are obtained by taking all words of length $n$ that are not subwords of the concatenation of two words of $\widetilde L_k$ for some $k\geq0$ such that $\ell_k \geq n$. 

\begin{proof}[Proof of Lemma \ref{lemma:OnedimensionalConcatenatedSubshift}]
\quad

{\it Proof of item \ref{Item:OnedimensionalConcatenatedSubshift_1}.}
By  assumption every word in $\widetilde L_{k+1}$ is a concatenation of words in $\widetilde L_k$. Then the concatenated subshifts obviously satisfy  $\langle \widetilde L_{k+1} \rangle\subseteq\langle \widetilde L_k \rangle$.
	
{\it Proof of item \ref{Item:OnedimensionalConcatenatedSubshift_2}.}
Let $x \in \widetilde X$ and $n\geq1$. Then $x \in \langle \widetilde L_k \rangle$ for some $n \leq \ell_k$. Therefore  $x \in \Sigma^1(\widetilde{\mathcal{A},} \widetilde{\mathcal{F}}_n)$ and $\widetilde X \subseteq \Sigma^1(\widetilde{\mathcal{A}}, \widetilde{\mathcal{F}})$. Conversely let $x \in \Sigma^1(\widetilde{\mathcal{A}}, \widetilde{\mathcal{F}})$ and $k\geq0$. Define the interval
\begin{displaymath}
 I_k := \Big\llbracket 1 - \Big\lfloor \frac{\ell_k}{2} \Big\rfloor,\ell_k- \Big\lfloor \frac{\ell_k}{2} \Big\rfloor  \Big\rrbracket.
\end{displaymath}
For any $j \geq k$, as $x \in \Sigma^1(\widetilde{\mathcal{A}}, \widetilde{\mathcal{F}}_{\ell_j})$, $x|_{I_j}$ is a subword of the concatenation of two words of length $\ell_j$ of $\widetilde L_j$. As $\langle \widetilde L_j \rangle \subseteq  \langle \widetilde L_k \rangle$,  $x|_{I_j}$ is a subword of the concatenation of words of length $\ell_k$ of $\widetilde L_k$. Let $y_j \in \langle \widetilde L_k \rangle$  such that $y_j|_{I_j} = x|_{I_j}$. By compactness of $\langle \widetilde L_k \rangle$, the sequence $(y_j)_{j\geq0}$ admits an accumulation point $y=x \in \langle \widetilde L_k \rangle$. Therefore $x \in \widetilde X$.

{\it Proof of item \ref{Item:OnedimensionalConcatenatedSubshift_3}.} We have obviously
\[
\forall\, n \leq \ell_k, \ \mathcal{L}(\widetilde X,n) \subseteq \mathcal{L}(\langle \widetilde L_k \rangle,n).
\]
Conversely consider two words $u_n,v_n\in \widetilde L_n$. We want to show that the concatenation $w_n := u_n v_n$ is a subword of some $x \in \widetilde X$.  We may assume  that $w_n$ is a pattern of support $K_n :=\llbracket 1 -\ell_n, 2\ell_n-\ell_n \rrbracket$. We construct by induction a sequence of patterns $(w_m)_{m\geq n}$ of support $K_m = \llbracket a_m, b_m \rrbracket$, $b_m-a_m = 2 \ell_m-1$, such that
\begin{itemize}
\item $w_m$ is  equal to  the concatenation of two words of $L_m$,
\item $K_m \subseteq K_{m+1}$ and $w_{m+1}|{K_m}=w_m$,
\item if $m$ is even then $b_{m+1}> b_m$, if $m$ is odd then $a_{m+1} < a_m$.
\end{itemize}
Indeed, assume $m$ is even and $w_m$ has been constructed. Then, by hypothesis on the language $\widetilde L_{m}$, $w_m$ is the subword of the concatenation of two words of $\widetilde L_{m+1}$. Let $\widetilde w_{m+1} = \widetilde u_{m+1} \widetilde v_{m+1}$ be the corresponding pattern of support $\widetilde K_{m+1} \supseteq K_m$ of length $2 \ell_{m+1}$ containing $w_m$. If $b_{m+1} > b_m$ we choose $K_{m+1}=\widetilde K_{m+1}$ and $w_{m+1} = \widetilde w_{m+1}$. If $b_{m+1}=b_m$, as $\ell_{m+1} \geq 2 \ell_m$ ($\ell_{m+1}> \ell_m$ and $\widetilde L_{m+1}$ is obtained by concatenating words of $\widetilde L_m$), then $w_m$ is a subword of the rightmost word of $\widetilde w_{m+1}$, that is $w_n \sqsubset \widetilde v_{m+1}$. We choose any word $v_{m+1}$ in $\widetilde L_{m+1}$, define the pattern $w_{m+1}=\widetilde v_{m+1}v_{m+1}$, and call $K_{m+1}$ the corresponding support. In the case $m$ is odd we do the analogous construction but on the left-hand side.

Let $x$ be the configuration such that $x| K_m = w_m$ for every $m\geq n$. Let $y_m \in \langle L_m \rangle$ be a configuration such that $w_m = y_m | K_m$. It follows that the sequence $(y_m)_{m\geq n}$ admits an accumulation point $y \in \widetilde X$ which satisfies $x|K_m = y|K_m$ for every $m\geq n$, and therefore $x=y \in \widetilde X$. This shows that $w_n\in\mathcal{L}(X, 2\ell_n)$.
\end{proof}

It is clear from the above arguments that $\Sigma^1(\widetilde{A},\widetilde{\mathcal{F}})$ is an effectively closed subshift. The following lemma shows that $\widetilde{\mathcal{F}}$ satisfies items (\ref{item:ReconstructionFunctionEstimate_1})--(\ref{item:ReconstructionFunctionEstimate_3}) of Proposition \ref{proposition:ReconstructionFunctionEstimate}. 

\begin{lemma}
\label{proposition:algorithm}
The following holds:
\begin{enumerate}
\item \label{Item:algorithm_1} The reconstruction function satisfies $R^{\widetilde X}(n)=n$.
\item \label{Item:algorithm_2} For every $n\geq0$, there exist unique integers $k\geq1$ and $N_k \geq p\geq2$ satisfying   
	\[
	\ell_{k-1} < n \leq  \ell_k \ \ \text{and} \ \ (p-1)\ell_{k-1} < n \leq p \ell_{k-1}.
	\]
If $(N_k-1)\ell_{k-1} < n \leq N_k \ell_{k-1}$, define $\widetilde{\mathcal{F}}'(n)=\widetilde{\mathcal{F}}(n)$. If $n \leq (N_k-1) \ell_{k-1}$, define $\widetilde{\mathcal{F}}'(n)$ as the set of words of length $n$ that are not subwords of any word of the form $\overrightarrow{w_1}\overleftarrow{w_2}$ where $\overrightarrow{w_1}$ is a terminal segment of $w_1$ of length $(p+1)\ell_{k-1}$, $\overleftarrow{w_2}$ is an initial segment of $w_2$  of length $(p+1)\ell_{k-1}$, and $w_1$ or $w_2$ are  either one of the words $a_k,b_k,1_k,2_k$. Then
\[
\forall\, n \in \llbracket 1, \ell_k \rrbracket, \ \widetilde{\mathcal{F}}'(n)=\widetilde{\mathcal{F}}(n).
\]
	
\item \label{Item:algorithm_3} There exists a Turing machine $\widetilde{\mathbb{M}}$ such that the patterns of $\widetilde{\mathcal{F}}$ are enumerated in increasing order, and such that there is a polynomial $P(n)$ such that the time enumeration function satisfies $T^{\widetilde X}(n)\leq P(n)|\widetilde{\mathcal{A}}|^n$ for every $n \geq 0$.
\end{enumerate}
\end{lemma}

 \begin{proof}
\quad
{\it Proof of item (\ref{Item:algorithm_1}).} A word $w$ of length $n$ that is not in $\widetilde{\mathcal{F}}$ is a subword of the concatenation of two words of length $\ell_k$ of $\widetilde L_k$. Lemma \ref{lemma:OnedimensionalConcatenatedSubshift} shows that $w \in \mathcal{L}(\widetilde X,n)$.

{\it Proof of item (\ref{Item:algorithm_2}).} We may assume $n \leq (N_k-1)\ell_{k-1}$ and  $p<N_k$. We have obviously $\widetilde{\mathcal{F}}(n) \subseteq \widetilde{\mathcal{F}}'(n)$. If we assume that $k$ is even, from Definition~\ref{notation:BaseLanguageBis} we have that
\[a_k = a_{k-1}(1^{\ell_{k-1}})^{N_k-2}a_{k-1}\quad \mbox{and}\quad b_k = (b_{k-1})^{N_k}.\]

We set
\[\overleftarrow{a_k}(1) = \overrightarrow{a_k}(1) = a_{k-1},\quad \overleftarrow{b_k}(1) = \overrightarrow{b_k}(1) = b_{k-1},
\]
\[
\overleftrightarrow{1_k}(1) = 1_{k-1}:=1^{\ell_{k-1}} \ \mbox{and} \ \overleftrightarrow{2_k}(1) = 2_{k-1}:=2^{\ell_{k-1}}.\]
Then we define by induction: if $2 \leq p<N_k$ then
\[\overleftarrow{a_k}(p) = \overleftarrow{a_k}(p-1) 1_{k-1} = a_{k-1}(1_{k-1})^{p-1}\]
and
\[\overrightarrow{a_k}(p)= 1_{k-1}\overrightarrow{a_k}(p-1) = (1_{k-1})^{p-1} a_{k-1},\]
else \ $\overleftarrow{a_k}(N_k) = \overrightarrow{a_k}(N_k) = a_k$. We also define
\[\overleftarrow{b_k}(p)= \overleftarrow{b_k}(p-1) b_{k-1} = (b_{k-1})^{p},\]
\[\overrightarrow{b_k}(p) = b_{k-1}\overrightarrow{b_k}(p-1) = (b_{k-1})^{p},\]
\[\overleftrightarrow{1_k}(p)=  \overleftrightarrow{1_k}(p-1)1_{k-1}= (1_{k-1})^{p}\]
and
\[\overleftrightarrow{2_k}(p)=  \overleftrightarrow{2_k}(p-1)1_{k-1}= (2_{k-1})^{p}.\]

If $w$ has length less than $p\ell_{k-1}$ and is a subword of some $w_1w_2$, say $w_1=a_k$ and $w_2=b_k$, by dragging $w$ from the left end point of $w_1w_2$ to the right end point of $w_1w_2$, the word $w$ appears  successively as a subword of $\overleftarrow{a_k}(p+1)$, $\overleftrightarrow{1_k}(p+1)$, $\overrightarrow{a_k}(p+1)$,  $\overrightarrow{a_k}(p+1)\overleftarrow{b_k}(p+1)$, $\overleftarrow{b_k}(p+1)$.  A similar reasoning is also true for $w_1=b_k$ and $w_2=a_k$. We have shown that $\widetilde{\mathcal{F}}(n) =\widetilde{\mathcal{F}}'(n)$.

{\it Proof of item (\ref{Item:algorithm_3}).} To compute the time to enumerate successively the words of $\widetilde{\mathcal{F}}(n) $ when $\ell_{k-1} < n  \leq \ell_k$, we can produce instead an algorithm which enumerates $\widetilde{\mathcal{F}}'$. The time to read/write on the tapes, to update the words $(\overleftarrow{a_k}(p),\overrightarrow{a_k}(p),\overleftarrow{b_k}(p),\overrightarrow{b_k}(p)$, $\overleftrightarrow{1_k}(p),\overleftrightarrow{2_k}(p))$ by adding a word of length $\ell_{k-1}$, to concatenate two words $\overrightarrow{w_1}\overleftarrow{w_2}$ from that list, and to check that a given word $w$ of length $n$ is a subword of $\overrightarrow{w_1}\overleftarrow{w_2}$ is polynomial in $n$. Therefore, the time to enumerate every word up to length $n$ in an alphabet $\widetilde{\mathcal{A}}$ is bounded by $P(n)|\widetilde{\mathcal{A}}|^n$ where $P(n)$ is some fixed polynomial.
\end{proof}

\subsection{The intermediate dictionaries}

We shall now study the complexity of the set of words $\mathcal{L}(\widetilde X, \ell_k')$ of length $\ell'_k = N'_k \ell_{k-1}$.

\begin{definition} \label{notation:BaseLanguageTer}
Let $\widetilde A'_k$ and $\widetilde B'_k$  be the sub-dictionaries of $\widetilde A_k$ and $\widetilde B_k$ that are made of subwords of length $\ell'_k$ that are either initial or terminal words of a word in $\widetilde A_k$ and $\widetilde B_k$. Formally,
\begin{enumerate}
\item if $k$ is odd, $\widetilde A'_k = \{ a'_k, 1^{\ell'_k}\}$, $\widetilde B'_k = \{b'_k,b''_k,2^{\ell'_k}\}$,
		\begin{equation}
 a'_k := \underbrace{a_{k-1} \cdots a_{k-1}}_{\text{$N'_k$ times}}, \quad
b'_k := b_{k-1} 2^{(N'_k-1)\ell_{k-1}}, \quad
b''_k := 2^{(N'_k-1)\ell_{k-1}}b_{k-1},
		\tag{R'1}\label{rule.odd.prime}
		\end{equation}
\item if $k$ is even, $\widetilde A'_k = \{ a'_k,a''_k,1^{\ell'_k}\}$, $\widetilde B'_k = \{ b'_k,2^{\ell'_k}\}$,
		\begin{equation}
a'_k := a_{k-1} 1^{(N'_k-1)\ell_{k-1}}, \\
a''_k = 1^{(N'_k-1)\ell_{k-1}}a_{k-1} \mbox{ and }\\
b'_k := \underbrace{b_{k-1}\cdots b_{k-1}}_{\text{$N'_k$ times}}. \\
\tag{R'2}
		\label{rule.even.prime}
		\end{equation}
\item $\widetilde{L}_k':=\widetilde{A}_k'\bigsqcup \widetilde{B}_k'$.
\end{enumerate}
\end{definition}

Notice that $\widetilde A'_k$ and $\widetilde B'_k$ have been chosen so that  the words of $\widetilde A_k$ (respectively $\widetilde B_k$) are obtained by concatenating $N_k/N'_k$ words of $\widetilde A'_k$ (respectively $\widetilde B'_k$). In particular we have
\[
\langle \widetilde A_k \rangle \subset \langle \widetilde A'_k \rangle \ \ \text{and} \ \ \langle \widetilde B_k \rangle \subset \langle \widetilde B'_k \rangle.
\]

We will say that two words $a,b \in \widetilde{\mathcal{A}}^\ell$ overlap if there exists a non-trivial shift $0 <s < \ell$ such that the terminal segment of  length $s$ of the word $a$ coincides with the initial segment of the word $b$ of the same length, or vice-versa by permuting $a$ and $b$. Note that we exclude the overlapping where $a$ and $b$ coincide.

The next three results are technical lemmas about the possible types of overlapping of words of $A'_k$ or $B'_k$. The first lemma asserts that there is no possible overlapping between words of $\widetilde A'_k$ and words of $\widetilde B'_k$. The next two lemmas characterize the possible overlaps between any two words at each stage $k$ of the iteration process.

\begin{lemma}
	\label{lemma:NoOverlappingAB}
	In our construction described above, a word from $\widetilde A_k'$ and a word from $\widetilde B_k'$ do not overlap. Similarly, a word from $\widetilde A_k$ and a word from $\widetilde B_k$ do not overlap.
\end{lemma}

\begin{proof}
	Every word in $\widetilde A_k'$ ends with the symbol $1$ which does not appear in any word in $\widetilde B_k'$. Conversely, every word in $\widetilde B_k'$ ends with the symbol $2$ that does not appear in any word in $\widetilde A_k'$. The same argument is valid for the words in $\widetilde A_k$ and $\widetilde B_k$.
\end{proof}

The next lemma is formulated for the case $k$ even, but a similar lemma holds for the case $k$ odd. First we need to fix some notations. Consider  the rules described in (\ref{rule.even}) and (\ref{rule.even.prime}). The {\it initial segment}  of $a_k$ and $a'_k$, the {\it terminal segment}   of $a_k$ and $a''_k$, and the {\it marker}  are the following subwords, for $k$ even
\begin{gather*}
a_k = \underbrace{a_{k-1}}_{a_{k-1}^I} 1^{(N_k-2)\ell_{k-1}} \underbrace{a_{k-1}}_{a_{k-1}^T}, \quad b'_k = \underbrace{b_{k-1}}_{b_{k-1}^I} b_{k-1}^{(N'_k-2)}\underbrace{b_{k-1}}_{b_{k-1}^T}, \\
a_k' = \underbrace{a_{k-1}}_{a_{k-1}^I} \underbrace{1^{(N_k'-1)\ell_{k-1}}}_{\mbox{marker}} \quad a_k'' = \underbrace{1^{(N_k'-1)\ell_{k-1}}}_{\mbox{marker}} \underbrace{a_{k-1}}_{a_{k-1}^T}.
\end{gather*}
Note that $a_{k-1}^I=a_{k-1}^T=a_{k-1}$ and $b_{k-1}^I=b_{k-1}^T=b_{k-1}$.

\begin{lemma}
	\label{lemma:OverlappingWords}
	Let $k\geq1$ be even, $a_k\in \widetilde A_k$ and $b_k\in \widetilde{B}_k$ as described in (\ref{rule.even}). Then
	\begin{enumerate}
		\item \label{item:OverlappingWords_1} two words of the same type $a_k$ can only overlap on their initial and terminal segment, that is, the segment $a_{k-1}^I$ of one of the two words overlaps the segment $a_{k-1}^T$ of the other word $a_k$;
		
		\item \label{item:OverlappingWords_2}  on the other hand, two words of the same type $b_k$ overlaps itself exactly on a power of $b_{k-1}$ or they have an overlap of length $\ell_{k-2}$ between $b_{k-1}^I$ and $b_{k-1}^T$.
	\end{enumerate}
\end{lemma}

\begin{proof}
{\it Proof of item (\ref{item:OverlappingWords_1}).} Consider a non-trivial shift $0 < s < \ell_{k}$ and a word $w \in \widetilde{\mathcal{A}}^{\llbracket 1, s+\ell_k \rrbracket}$ made of two overlapping $a_k$:
		\[
		a_k = w|_{\llbracket 1, \ell_k \llbracket}, \quad \widetilde a_k := w|_{s+\llbracket 1,\ell_k \rrbracket}, \quad \forall\, i \in \llbracket1,\ell_k\rrbracket, \ \widetilde a_k(s+i) = a_k(i).
		\]
		
		We assume first that $0 < s < \ell_{k-1}$. On the one hand $a_{k-1}^T$ of $a_k$ starts with the symbol $0$ at the index $i=(N_k-1)\ell_{k-1}+1$. On the other hand the symbol $1$ appears in $\widetilde a_k$ at the indexes in the range $\llbracket \widetilde i, \widetilde j \rrbracket := \llbracket s+\ell_{k-1}+1,s+(N_k-1)\ell_{k-1} \rrbracket$. Since $i \in \llbracket \widetilde i, \widetilde j \rrbracket$ we obtain a contradiction. 
		
		We assume next that $\ell_{k-1} \leq s < (N_k-1)\ell_{k-1}$. On the one hand the symbol $1$ appears in $a_k$ at the indexes in the range $\llbracket \widetilde i, \widetilde j \rrbracket := \llbracket \ell_{k-1}+1,(N_k-1)\ell_{k-1}\rrbracket$. On the other hand $\widetilde a_k$ starts with the symbol $0$ at the index $i=s+1$. We obtain again a contradiction.
		
		We conclude that $s$ should satisfy $s \geq (N_k-1)\ell_{k-1}$: two words of the form $a_k$ can only overlap  on their initial and terminal segments.
		
{\it Proof of item (\ref{item:OverlappingWords_2}).}  We notice that $k-1$ is odd and $b_{k-1}$ has the same structure as $a_k$ in the first item. Two words of the form $b_{k-1}$ only overlap on their initial and terminal segments.  Then $b_{k-1}$ cannot be a subword of the concatenation $c=b_{k-1}b_{k-1}$ of two words $b_{k-1}$ unless $b_{k-1}$ coincides with the first or the last $b_{k-1}$ in $c$. If $b_k$ and $\widetilde b_k$ overlap, either $\widetilde b_k$ has been shifted by a multiple of $\ell_{k-1}$, $s \in \{ \ell_{k-1}, 2 \ell_{k-1}, \ldots, (N'_k-1)\ell_{k-1} \}$. Note that $k-1$ is an odd number, and so $b_{k-1}$ has the same behavior of $a_k$ described in the previous item. Therefore, it is only possible to have an overlap of a word $b_{k-2}$ of length $\ell_{k-2}$ between $b_{k-1}^T$ and $\widetilde b_{k-1}^I$.\end{proof}

\begin{lemma}
	\label{lemma:OverlappingWordsPrime}
	Let $k\geq1$ be an even integer and $a_k'$ and $a_k''$ as described in (\ref{rule.even.prime}). Then the following holds:
	\begin{enumerate}	
		\item \label{item:OverlappingWordsPrime_1} two words of the same form $a'_k$ never overlap; the same is true for two words of the same form $a_k''$;
		
		\item \label{item:OverlappingWordsPrime_2} two words $a'_k$ and $a''_k$ overlap if and only if they overlap either partially on their marker or partially on their initial and terminal segments, respectively. 
	\end{enumerate}
\end{lemma}

\begin{proof}
{\it Proof of item~(\ref{item:OverlappingWordsPrime_1}).} We consider a non-trivial shift $0 < s < \ell'_k$ and two overlapping words of the form $a'_k$ shifted by $s$. Let $w \in {\widetilde{\mathcal{A}}}^{\llbracket 1, s+\ell'_k \rrbracket}$ such that
		\begin{displaymath}
		\dis a'_k = w|_{\llbracket 1, \ell'_k \rrbracket}, \quad \widetilde a'_k := w|_{s+\llbracket 1,\ell'_k \rrbracket}, \quad \forall\, i \in \llbracket1,\ell'_k\rrbracket, \ \widetilde a'_k(s+i) = a'_k(i).
		\end{displaymath}
		
		We assume first that $\ell_{k-1} \leq s < \ell'_{k}$. On the one hand, $\widetilde a'_k$ starts with the symbol $0$, $w(s+1)=0$; on the other hand, $w|_{\llbracket \ell_{k-1}+1, \ell'_k \rrbracket}$ contains only the symbol $1$. Since $s+1 \in \llbracket \ell_{k-1}+1, \ell'_k \rrbracket$ we obtain a contradiction. 
		
		We assume next that $0 < s < \ell_{k-1}$. We observe that $k-1$ is odd and the two initial segments $a^I_{k-1}$ of $a'_k$ and $\widetilde a'_k$ are of the same form as $b_k$ in the second item. They overlap on a multiple of words of the form $a_{k-2}$ or at their initial and terminal segments. Necessarily $s \geq l_{k-2}\geq2$. On the one hand, the initial segment of $\widetilde a'_k$ ends with the symbols $01$, $w(s+\ell_{k-1}-1)=0$, on the other hand, $w|_{\llbracket\ell_{k-1}+1,\ell'_k \rrbracket}$ contains only the symbol $1$. Since $s+\ell_{k-1}-1 \in \llbracket \ell_{k-1}+1, \ell'_k \rrbracket$ we obtain a contradiction.  A similar proof works for $a''_k$ instead of $a'_k$.
		
{\it Proof of item (\ref{item:OverlappingWordsPrime_2}).} We divide our discussion in two cases. Consider first,
		\begin{displaymath}
		\dis a'_k = w|_{\llbracket 1, \ell'_k \rrbracket}, \quad \widetilde a''_k := w|_{s+\llbracket 1,\ell'_k \rrbracket}, \quad \forall\, i \in \llbracket1,\ell'_k\rrbracket, \ \widetilde a''_k(s+i) = a''_k(i).
		\end{displaymath}
		Suppose that $0 < s < \ell_{k-1}$. On the one hand the terminal segment of $\widetilde a''_k$ is a word like $a_{k-1}$ and then it starts with the symbol $0$ which appears in $w$ at the index $s+(N'_k-1)\ell_{k-1} \in \llbracket \ell_{k-1},\ell'_k\rrbracket$. On the other hand $w|_{\llbracket \ell_{k-1}, \ell'_k \rrbracket}$ contains only the symbol $1$. Thus we obtain a contradiction. We conclude that necessarily $\ell_k \leq s$ and the two words $a'_k$ and $a''_k$ overlap (partially or completely) on their markers.
		
		We consider next the case,
		\begin{displaymath}
		\dis a''_k = w|_{\llbracket 1, \ell'_k \rrbracket}, \quad \widetilde a'_k := w|_{s+\llbracket 1,\ell'_k \rrbracket}, \quad \forall\, i \in \llbracket1,\ell'_k\rrbracket, \ \widetilde a'_k(s+i) = a'_k(i).
		\end{displaymath}
		Suppose that $0 < s < (N'_k-1)\ell_{k-1}$. On the one hand the initial segment of $\widetilde a'_k$ starts with the symbol $0$ which is located at the index $s+1 \in \llbracket 1, (N'_k-1)\ell_{k-1} \rrbracket$ in $w$. On the other hand $w|_{\llbracket 1,(N'_k-1)\ell_{k-1}\rrbracket}$ is the marker of $a''_k$ and contains only the symbol $1$. We obtain a contradiction. We conclude that it is only possible to have $s \geq  (N'_k-1)\ell_{k-1}$, which means that the terminal segment of $a''_k$ overlaps with the initial segment of $a'_k$. Both segments are copies of $a_{k-1}$ and as we are in the case where $k\geq 2$ is even, we have that $k-1$ is odd and thus $a_{k-1}$ has the same behavior described in Lemma~\ref{lemma:OverlappingWords} item (2). Therefore the possible overlap can occur (partially or completely) on their initial and terminal segments by the rules described as in Lemma~\ref{lemma:OverlappingWords} item (2).
\end{proof}

\subsection{The vertically aligned subshift}

We estimate the entropy of the Gibbs measure $\mu_{\beta_k}$ from above in Lemma \ref{lemma:EntropyEstimate_2} by measuring the frequency of the duplicated  symbol $0' ,0''$. If $k$ is even  most of the symbols $0$ are in the words $b_k\in \widetilde B_k$. As $b_k$ does not use the symbol $1$, a positive frequency of the symbol $1$ for a generic $\mu_{\beta_k}$-configuration tends to decrease the occurrence of the words $b_k$. The purpose of this section is to quantitatively justify this intuition.

\begin{definition}
Let $\widetilde{\widetilde A}'_{k}\subseteq\widetilde{\mathcal{A}}^{\llbracket 1,\ell'_k\rrbracket^2}$ be the bidimensional dictionary of length $\ell'_k$ of vertically aligned patterns that project onto $\widetilde L'_k$, formally defined as
\begin{displaymath}
\widetilde{\widetilde A}'_{k} := \big\{p \in \widetilde{\mathcal{A}}^{\llbracket 1, \ell'_k \rrbracket^2} : \exists \widetilde p \in  \widetilde A'_k, \ \text{s.t.} \ \ \forall, (i,j) \in\llbracket 1, \ell'_k \rrbracket^2, \ p(i,j) = \widetilde p(i) \big\}.
	\end{displaymath}
The dictionary $\widetilde{\widetilde B}'_k \subseteq\widetilde{\mathcal{A}}^{\llbracket 1,\ell'_k\rrbracket^2}$ is defined analogously. Let $\widetilde{\widetilde X}$ be the set of vertically aligned configurations that project onto $\widetilde X$
\[
\widetilde{\widetilde X} := \big\{ x \in \widetilde{\mathcal{A}}^{\mathbb{Z}^2} : \exists \widetilde x \in \widetilde X,  x(i,j) = \widetilde x(i) \mbox{ for every } (i,j) \in \mathbb{Z}^2 \big\}.
\]

We use the notation $\widetilde{\widetilde \pi}
\colon \widetilde{\widetilde X} \to \widetilde X$ or $\widetilde{\widetilde \pi} \colon \widetilde{\widetilde A}'_k \to \widetilde A'_k$ to represent the projection of a vertically aligned configuration or pattern.
\end{definition}

Let $p \in \widetilde{\mathcal{A}}^{\llbracket 1,n\rrbracket^2}$ be a large pattern (not necessarily vertically aligned) and consider the set of translates $u$ of small squares of size $2\ell'_k$ inside this pattern $p$ that are vertically aligned and project onto a pattern of ${\widetilde A}'_k$ or ${\widetilde B}'_k$.  We introduce the following notations.

\begin{definition} \label{Definition:IJK}
Let $k \geq 2$,  $n > 2\ell'_k$, and  $p \in \widetilde{\mathcal{A}}^{\llbracket 1,n\rrbracket^2}$. We define
\begin{enumerate}
\item \label{Item:IJK_1}  $\displaystyle I(p,\ell'_k):=\left\{u\in\llbracket 0,n-2\ell'_k\rrbracket^2:\sigma^u(p)|_{\llbracket 1,2\ell'_k\rrbracket^2}\in\mathcal{L}(\widetilde{\widetilde X}, 2\ell'_k)\right\}$,
\item \label{Item:IJK_2}  $\displaystyle I^A(p,\ell'_k):=\left\{u\in\llbracket 0,n-\ell'_k\rrbracket^2:\sigma^u(p)|_{\llbracket 1,\ell'_k \rrbracket^2}\in \widetilde{\widetilde A}'_{k}\right\}$,
\item \label{Item:IJK_3}  $\displaystyle J^A(p,\ell'_k):=\bigcup_{u\in I^A(p,\ell'_k)}\left(u+\llbracket 1,\ell'_k \rrbracket^2\right)$.
\end{enumerate}
We define $I^B(p,\ell'_k)$ and $J^B(p,\ell'_k)$ similarly with replacing $\widetilde{\widetilde A}_{k}'$ for $\widetilde{\widetilde B}'_{k}$ in (\ref{Item:IJK_2}) and (\ref{Item:IJK_3}), respectively.
\end{definition}

\begin{lemma}
\label{lemma:AdmissibilityIntermediateScale}
Let $k \geq 2$,  $n > 2\ell'_k$,  $p \in \widetilde{\mathcal{A}}^{\llbracket 1,n\rrbracket^2}$ and $ I(p,\ell'_k),J^A(p,\ell'_k),J^B(p,\ell'_k)$  as in Definition \ref{Definition:IJK}. Let $\tau'_k=:(\ell'_k,\ell'_k)\in\mathbb{N}^2$. Then $J^A(p,\ell'_k)\cap J^B(p,\ell'_k)=\varnothing$ and
\[ 
\tau'_k + I(p,\ell'_k) \subset J^A(p,\ell'_k) \sqcup J^B(p,\ell'_k).  
\]
\end{lemma}

\begin{proof}
The fact that $J^A(p,\ell'_k)$ and $J^B(p,\ell'_k)$ do not intersect is a consequence of Lemma \ref{lemma:NoOverlappingAB}. Let $u \in I(p,\ell'_k)$ and  $w_* = \sigma^u(p)|_{\llbracket 1,2\ell'_k \rrbracket^2}$. There exists $w\in \mathcal{L}(\langle \widetilde L_k \rangle,2\ell'_k)$  such that $w_*(i,j)=w(i)$ for all $(i,j) \in \llbracket 1,2\ell'_k \rrbracket^2$. By definition of $\langle \widetilde L_k \rangle$, $w \sqsubset w_1w_2$ is a subword of the concatenation of two words of $\widetilde L_k$. By Definition \ref{notation:BaseLanguageTer} we have that $\langle\widetilde{L}_k\rangle\subseteq \langle\widetilde{L}_k'\rangle$, and also $\mathcal{L}(\langle\widetilde{L}_k\rangle,2\ell_k')\subseteq \mathcal{L}(\langle\widetilde{L}_k'\rangle,2\ell_k')$.
	
On the other hand, a word in $\widetilde L_k$ is either a word of $\widetilde A_k$ or a word of $\widetilde B_k$. As $\langle \widetilde A_k \rangle \subset \langle \widetilde A'_k \rangle$ and $\langle \widetilde B_k \rangle \subset \langle \widetilde B'_k \rangle$, $w_1$ and $w_2$ are obtained as a concatenation of words of $\widetilde A'_k$ or $\widetilde B'_k$. There exists $0 \leq s < \ell'_k$ such that 
	\begin{displaymath}
	\dis \sigma^s(w)|_{\llbracket 1, \ell'_k \rrbracket} \in \widetilde A'_k \bigsqcup \widetilde B'_k.
	\end{displaymath}
Then
	\begin{displaymath}
	u+(s,s) \in I^A(p,\ell'_k) \bigsqcup I^B(p,\ell'_k),
	\end{displaymath}
and therefore
	\begin{displaymath}
	\dis u+ \tau'_k  \in J^A(p,\ell'_k) \bigsqcup J^B(p,\ell'_k).
	\end{displaymath}
	This concludes the proof. See Figure~\ref{fig:diagram1008} for an illustration of this result.
\end{proof}

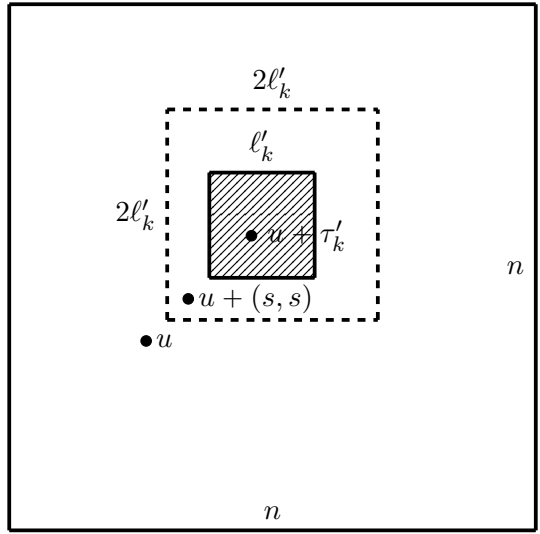
\begin{figure}[hbt]
\centering
\begin{tikzpicture}[scale=0.7]
\draw[line width = 0.5mm] (0,0) -- (0,10);
\draw[line width = 0.5mm] (0,0) -- (5,0) node[above] {$n$} -- (10,0) ;
\draw[line width = 0.5mm] (10,0) -- (10,5)  node[left] {$n$} -- (10,10);
\draw[line width = 0.5mm] (0,10) -- (10,10);
\draw[line width = 0.5mm, dashed] (3,4) -- (7,4);
\draw[line width = 0.5mm, dashed] (3,4) -- (3,6) node[left]{$2 \ell'_k$} -- (3,8);
\draw[line width = 0.5mm, dashed] (3,8) -- (5,8) node[above]{$2 \ell'_k$} -- (7,8);
\draw[line width = 0.5mm, dashed] (7,4) -- (7,8);
\filldraw (4.6,5.6) circle (1mm)  node[right, xshift=0.08cm] {$u+\tau'_k$};
\filldraw (2.6,3.6) circle (1mm) node[right] {$u$};
\draw[line width = 0.5mm] (3.8,4.8) -- (3.8,6.8);
\draw[line width = 0.5mm] (3.8,4.8) -- (5.8,4.8);
\draw[line width = 0.5mm] (5.8,4.8) -- (5.8,6.8);
\draw[line width = 0.5mm] (3.8,6.8) -- (4.8,6.8) node[above]{$\ell'_k$} -- (5.8,6.8);
\draw[pattern = {north east lines}, pattern color = black] (3.8,4.8) rectangle (5.8,6.8);
\filldraw (3.4,4.4) circle (1mm) node[right] {$u+(s,s)$};
\end{tikzpicture}
\caption{The biggest square is the square pattern $p\in \widetilde{\mathcal{A}}^{\llbracket 0,n\rrbracket^2}$. Let $u\in I(p,\ell'_k)$. The dashed square of size $2\ell'_k$ is a pattern in   $\mathcal{L}(\widetilde{\widetilde X}, 2\ell'_k)$. The pattern located in the innermost square of size $\ell'_k$ belongs to $\widetilde A'_k \bigsqcup \widetilde B'_k$. The innermost dot represents $u+\tau'_k$.} \label{fig:diagram1008}
\end{figure}

%\begin{figure}[hbt]
%\centering
%\includegraphics[width=0.8\linewidth]{diagram1008.pdf}
%\caption{The biggest square is the square pattern $p\in \widetilde{\mathcal{A}}^{\llbracket 0,n\rrbracket^2}$. Let $u\in I(p,\ell'_k)$. The dashed square of size $2\ell'_k$ is a pattern in   $\mathcal{L}(\widetilde{\widetilde X}, 2\ell'_k)$. The pattern located in the innermost square of size $\ell'_k$ belongs to $\widetilde A'_k \bigsqcup \widetilde B'_k$. The innermost dot represents $u+\tau'_k$.}
%	\label{fig:diagram1008}
%\end{figure}

\begin{lemma}
	\label{lemma:FrequencyOf0}
	Let $k\geq2$ be an even integer, $n>2\ell'_k$, and $ p \in \widetilde{\mathcal{A}}^{\llbracket 1,n \rrbracket^2}$. Let $I^A(p,\ell'_k)$, $J^A(p,\ell'_k)$, $I^B(p,\ell'_k)$, $J^B(p,\ell'_k)$ as in Definition \ref{Definition:IJK}. Define
	\begin{equation*}
	\label{eq.KA}
	K^A(p,\ell'_k) = \big\{ v \in J^A(p,\ell'_k) : p(v) = 0 \big\}, \ \ K^B(p,\ell'_k) = \big\{ v \in J^B(p,\ell'_k) : p(v) = 0 \big\}.
	\end{equation*}
	Then
	\begin{enumerate}
		\item $\displaystyle \Card(K^B(p,\ell'_k)) \leq \big(1-N_{k-1}^{-1}\big)^{-1} \Card(J^B(p,\ell'_k)) f_{k-1}^B$,
		\item $\displaystyle\Card(K^A(p,\ell'_k)) \leq \frac{2}{N'_k}\Card(J^A(p,\ell'_k))f_{k-1}^A$.
	\end{enumerate}
\end{lemma}

\begin{proof}
	Let $k\geq 2$ even, $n>2\ell_k'$ and a fixed $p\in\widetilde{\mathcal{A}}^{\llbracket 1,n \rrbracket^2}$. To simplify the notations, we write $I^A= I^A(p,\ell'_k)$, $J^A=J^A(p,\ell'_k)$ and so on. As the symbol $0$ does not appear in the markers $1^{N'_k\ell_{k-1}}\in \widetilde{A}_k'$ and $2^{N'_k\ell_{k-1}}\in\widetilde{B}_k'$, we only need to consider translates  $u \in \llbracket 0, n-\ell'_k \rrbracket^2$ of $I^A$ (resp. $I^B$) such that $w_*=\sigma^u(p)|_{\llbracket 1,\ell'_k\rrbracket^2}$ satisfying $\widetilde{\widetilde \pi}(w_*)\in\{a'_k,a''_k\}$ (resp. $\widetilde{\widetilde \pi}(w_*)=b'_k$).
	
	{\it Item 1.} We first enumerate $I^B = \{u_1,u_2,\ldots,u_H\}$.  Let $u_h= (u_h^x,u_h^y) \in\mathbb{Z}^2$. Let 
	\begin{displaymath}
	\dis J^B:=\bigcup_{h=1}^H J_h\quad\text{where}\quad J_h:=u_h+\llbracket 1,\ell'_k\rrbracket^2, \ \  \widetilde{\widetilde\pi}(\sigma^{u_h}(p))|_{\llbracket1,\ell'_k\rrbracket^2}=b'_k,
	\end{displaymath}
	that is, we are only considering the $J_h$ squares of $J^B(p,\ell_k')$ that contains vertically aligned word $b_k'$. For each box $J_h$, we divide it into $N_k'$ vertical strips of length $\ell_{k-1}$. Formally we have
	\begin{displaymath}
	\dis J_h=\bigcup_{i=1}^{N'_k}J_{h,i}\quad\text{where}\quad J_{h,i}:=u_h+\llbracket 1+(i-1)\ell_{k-1},i\ell_{k-1}\rrbracket\times\llbracket1,\ell'_k \rrbracket.
	\end{displaymath}
We construct a partition of $J^B$  inductively by,
	\begin{gather*}
	J^B=\bigsqcup_{h=1}^H J_h^*,\quad J_1^*= J_1, \ \ \forall\, h \geq2, \ J_h^*:=J_{h}\setminus\left(J_1\cup\cdots\cup J_{h-1}\right).
	\end{gather*}
	Let
	\begin{displaymath}
	\dis K_h^*:=\{v\in J_h^*:p(v)=0\}, \ K^B:=\bigsqcup_{h=1}^H K_h^*.
	\end{displaymath}
	It will be enough to show that for every $h \in \llbracket 1,H \rrbracket$
	\begin{equation}
	\label{eq.Card.Kh.B}
	\dis \Card( K_h^*) \leq \big(1-N_{k-1}^{-1}\big)^{-1}  \Card(J_h^*) f_k^B.
	\end{equation}
	By definition of $u_h$, $\widetilde w_h = \widetilde{\widetilde \pi}(p|_{(u_h+\llbracket 1, \ell'_k \rrbracket^2)})$= $b'_k\in\widetilde{\mathcal{A}}^{\ell_k'}$,
	\begin{displaymath}
	\dis \forall\,i,j\in \llbracket 1,\ell_k\rrbracket^2, \ \widetilde{w}_h(u_h^x+i)=b'_k(i).
	\end{displaymath}
	Since $b'_k$ can be decomposed into $N'_k$ subwords of the form $b_{k-1}$, we denote by $\widetilde{w}_{h,i}\in\widetilde{\mathcal{A}}^{\ell_{k-1}}$ the successive subwords for every $1\leq i\leq N'_k$. Formally,
	\begin{displaymath}
	\dis \widetilde w_{h,i}:=\widetilde w_h|_{(u_h^x+\llbracket 1+(i-1)\ell_{k-1},i\ell_{k-1}\rrbracket)} \mbox{ and } \sigma^{u_h^x+(i-1)\ell_{k-1}}(\widetilde w _{h,i}) = b_{k-1}.
	\end{displaymath}
	
	Consider now a fixed position $h$. We will show that $J_h^*$ is equal to a disjoint union of $N'_k$ vertical strips $(J_{h,i}^*)_{i=1}^{N'_k}$ of the following forms:
	\begin{itemize}
		\item the initial strip $J_{j,1}^*$,
		\begin{displaymath}
		u_h+\left(\llbracket 1+\ell_{k-2},\ell_{k-1}\rrbracket\times\llbracket c_{h,1}, d_{h,1} \rrbracket\right) \subseteq  J_{j,1}^* \subseteq \left( u_h+\llbracket 1,\ell_{k-1} \rrbracket \right) \times \llbracket c_{h,1}, d_{h,1} \rrbracket;
		\end{displaymath}
		\item  the intermediate strips, $J_{h,i}^*$, $1<i<N'_k$,
		\begin{displaymath}
		J_{h,i}^* = u_h+\left(\llbracket(i-1)\ell_{k-1}+1,i\ell_{k-1}\rrbracket\times\llbracket c_{h,i}, d_{h,i}\rrbracket\right),
		\end{displaymath}
		\item the terminal strip $J_{h,N'_k}^*$,
		\begin{multline*}
		u_h+\left(\llbracket 1+(N'_k-1)\ell_{k-1},\ell_k-\ell_{k-2} \rrbracket \times \llbracket c_{h,N'_k}, d_{h,N'_k} \rrbracket\right) \subseteq \\
		\subseteq  J_{h,N'_k}^* \subseteq 
		u_h+\left(\llbracket 1+(N'_k-1)\ell_{k-1},\ell'_k \rrbracket\times\llbracket c_{h,N'_k}, d_{h,N'_k} \rrbracket\right).
		\end{multline*}
	\end{itemize}
	Here for each $i\in\llbracket1,N_k'\rrbracket$, the values $1 \leq c_{h,i},d_{h,i}\leq \ell_k'$ are integers that represent the vertical length of each strip. Note that it possible that $c_{h,i} > d_{h,i}$, which denotes an empty strip $J_{h,i}^*$.
	
	Indeed, for a fixed $1 \leq i \leq N'_k$, we first consider the previous $J_g$, $1 \leq g < h$, that intersects the strip $J_{h,i}$   so that the word $\widetilde w_g$ overlaps $\widetilde w_h$ on a power of $b_{k-1}$ (see item (2) of Lemma \ref{lemma:OverlappingWords}). Then $c_{h,i}$ is the largest upper level of those $J_g \cap J_{h,i}$, more precisely,
	\begin{equation}
	\label{eq.c_hi}
	\dis c_{h,i} = \max_g \big\{ u_g^y + \ell_k' + 1 :  u_g^y \leq u_h^y, \big(u_h^x+(i-1) \ell_{k-1}+\llbracket 1, \ell_{k-1} \rrbracket \big) \subseteq \big( u_g^x + \llbracket 1, \ell_k' \rrbracket \big) \big\},
	\end{equation}
	and similarly $d_{h,i}$ is the smallest lower level of those $J_g \cap J_{h,i}$, formally we have
	\begin{equation}
	\label{eq.d_hi}
	\dis d_{h,i} =\min_g \big\{ u_g^y + 1 :  u_g^y \geq u_h^y, \big(u_h^x+(i-1) \ell_{k-1}+\llbracket 1, \ell_{k-1} \rrbracket \big) \subseteq \big( u_g^x + \llbracket 1, \ell_k' \rrbracket \big) \big\}.
	\end{equation}
	We have just constructed the intermediate strips $J_{h,i}^*$ for $1<i<N_k$, se Figure \ref{fig:diagram01}.
	
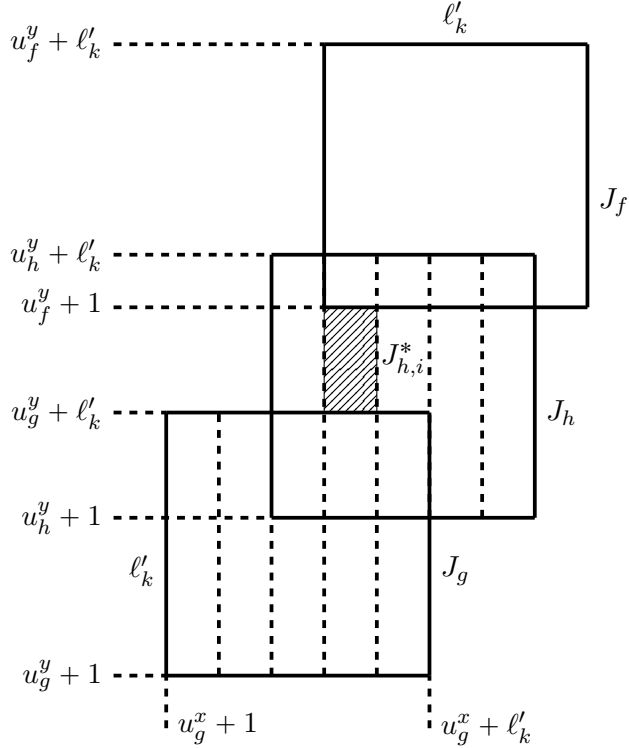
\begin{figure}[hbt]
\centering
\begin{tikzpicture}[scale=0.7]
\draw[line width = 0.5mm] (1,1)  -- (6,1);
\draw[line width = 0.5mm] (1,1) -- (1,3) node[left]{$\ell'_k$} -- (1,6);
\draw[line width = 0.5mm] (6,1) -- (6,3) node[right]{$J_g$} -- (6,6);
\draw[line width = 0.5mm] (1,6) -- (6,6);
\draw[line width = 0.5mm] (3,4) -- (8,4);
\draw[line width = 0.5mm] (3,4) -- (3,9);
\draw[line width = 0.5mm] (3,9) -- (8,9);
\draw[line width = 0.5mm] (8,4) -- (8,6) node[right]{$J_h$} -- (8,9);
\draw[line width = 0.5mm] (4,8) -- (9,8);
\draw[line width = 0.5mm] (4,8) -- (4,13);
\draw[line width = 0.5mm] (4,13) -- (6.5,13) node[above]{$\ell'_k$} -- (9,13);
\draw[line width = 0.5mm] (9,8) -- (9,10) node[right]{$J_f$} -- (9,13);
\draw[line width = 0.5mm, dashed] (4,4) -- (4,9);
\draw[line width = 0.5mm, dashed] (5,4) -- (5,7) node[right, xshift=-1mm]{$J_{h,i}^*$} -- (5,9);
\draw[line width = 0.5mm, dashed] (6,4) -- (6,9);
\draw[line width = 0.5mm, dashed] (7,4) -- (7,9);
\draw[line width = 0.5mm, dashed] (2,1) -- (2,6);
\draw[line width = 0.5mm, dashed] (3,1) -- (3,4);
\draw[line width = 0.5mm, dashed] (4,1) -- (4,4);
\draw[line width = 0.5mm, dashed] (5,1) -- (5,4);
\draw[pattern = {north east lines}, pattern color = black] (4,6) rectangle (5,8);
\draw[line width = 0.5mm, dashed] (0,1) node[left]{$u_g^y+1$} -- (1,1);
\draw[line width = 0.5mm, dashed] (0,4) node[left]{$u_h^y+1$} -- (3,4);
\draw[line width = 0.5mm, dashed] (0,6) node[left]{$u_g^y+\ell'_k$} -- (1,6);
\draw[line width = 0.5mm, dashed] (0,9) node[left]{$u_h^y+\ell'_k$} -- (3,9);
\draw[line width = 0.5mm, dashed] (0,8) node[left]{$u_f^y+1$} -- (4,8);
\draw[line width = 0.5mm, dashed] (0,13) node[left]{$u_f^y+\ell'_k$} -- (4,13);
\draw[line width = 0.5mm, dashed] (1,0) node[right]{$u_g^x+1$} -- (1,1);
\draw[line width = 0.5mm, dashed] (6,0) node[right]{$u_g^x+\ell'_k$} -- (6,1);
\end{tikzpicture}
\caption{We represent the intermediate strip $J^*_{j,i}$ in $J_h$ that is obtained after discarding $J_g$ and $J_f$ constructed before.}
\label{fig:diagram01}
\end{figure}
	
	We now construct the initial strip (the terminal strip is constructed similarly). We intersect the remaining $J_g$ with $J_{h,1}$. The terminal segment $b_{k-1}^T$ of $\widetilde w_g$ overlaps the initial segment $b_{k-1}^I$ of $\widetilde w_h$. By item (1) of Lemma \ref{lemma:OverlappingWords}, as $k-1$ is odd, $b_{k-1}$ has the same structure as $a_k$, and hence the overlapping can only happen at their end segments of the form $b_{k-2}$. We have just proved that $J_{h,1}^*$ contains a small strip $\big( u_h + \llbracket 1+\ell_{k-2},\ell_{k-1} \rrbracket \big) \times \llbracket c_{h,1}, d_{h,1} \rrbracket$ of base $b_{k-1}^I \setminus b_{k-2}$ and is included in a larger strip $\big( u_h+\llbracket 1,\ell_{k-1} \rrbracket \big) \times \llbracket c_{h,1},d_{h,1} \rrbracket$ of base $b_{k-1}$. For the initial and terminal strip the vertical extension ($\llbracket c_{h,1},d_{h,1}\rrbracket$ and $\llbracket c_{h,{N}_k'},d_{h,{N}_k'}\rrbracket$) of the elements $J_{h,1}^*$ and $J_{h,N_k'}^*$ are defined as in (\ref{eq.c_hi}) and (\ref{eq.d_hi}).

	Let $K_{h,i}^*:=\{v\in J_{h,i}^*:p_v=0\}$. We show that
	\begin{equation}\label{eq.K_hi}
	\Card( K_{h,i}^*) \leq  \big(1-N_{k-1}^{-1}\big)^{-1}  \Card( J_{h,i}^*) f_k^B \mbox{ for every }1 \leq i \leq N_k.
	\end{equation}
	
	For the intermediate strips $ J_{h,i}^*$, where $1 < i < N'_k$, we use the fact that $ J_{h,i}^*$ is a square strip of base $b_{k-1}$, and the fact that the frequency $f_{k-1}^B$ of the symbol $0$ in the word $b_{k-1}$ is identical to the frequency $f_k^B$ of the symbol $0$ in $b_k$. We have,
	\begin{displaymath}
	\Card(K_{h,i}^*) = \ell_{k-1}(d_{h,i}-c_{h,i}+1)  f_k^B =  \Card( J_{h,i}^*) f_k^B.
	\end{displaymath}
	
	For the initial strip $J_{h,1}^*$, we use the fact $J_{h,1}^*$ resembles largely a square strip of base $b_{k-1}$. We have,
	\begin{align*}
	\Card(\ K_{h,i}^*) &\leq  \ell_{k-1} (d_{h,1}-c_{h,1}+1) f_{k}^B \\
	&\leq \frac{\ell_{k-1}}{\ell_{k-1}-\ell_{k-2}} (\ell_{k-1}-\ell_{k-2}) (d_{h,1}-c_{h,1}+1)f_{k}^B \\
	&\leq \big(1-N_{k-1}^{-1}\big)^{-1} \Card( J_{h,i}^* ) f_k^B.
	\end{align*}
	We have proven (\ref{eq.K_hi}) and by summing over $i \in \llbracket 1, N'_k \rrbracket$ we have proven (\ref{eq.Card.Kh.B}).
	
	{\it Item 2.} As before we will consider $I^A$, but only consider the translates $u \in \llbracket 0, n-\ell'_k \rrbracket^2$ such that $\widetilde{\widetilde\pi}(\sigma^u(p)|_{\llbracket 1,\ell'_k\rrbracket^2})\in\{a'_k,a''_k\}$. If $J_g\cap J_h\neq\varnothing$, the two projected words $\widetilde w_g = \widetilde{\widetilde\pi}(\sigma^{u_g}(p)|_{\llbracket 1,\ell'_k \rrbracket^2})$ and $\widetilde w_h = \widetilde{\widetilde\pi}(\sigma^{u_h}(p)|_{\llbracket 1,\ell'_k \rrbracket^2})$ may coincide in three ways: either $u_g^x=u_h^x$ and $\widetilde w_g = \widetilde w_h$, or $\widetilde w_g$ and $\widetilde w_h$ intersect on their markers, or $\widetilde w_g$ and $\widetilde w_h$ intersect on their initial and terminal segments, as in Lemma \ref{lemma:OverlappingWordsPrime}.
	
	We redefine again $I^A$ by clustering into a unique rectangle formed by adjacent squares where the overlap occurs in the whole word, that is, we group the squares $J_g$ and $J_h$ that pairwise satisfy $J_g\cap J_h \not= \varnothing$, $u_g^x=u_h^x$, $\widetilde w_g = \widetilde w_h$ and $|u_g^y-u_h^y|< \ell'_k$. Then, after re-indexing $I^A$, one obtains,
	\begin{displaymath}
	\dis J^A = \bigcup_{h=1}^H J_h, \quad J_h =  u_h + \left( \llbracket 1, \ell'_k \rrbracket \times \llbracket 1, d_h \rrbracket \right),
	\end{displaymath}
	where $d_h$ is the final height of each rectangle obtained after the clustering. Thus $w_h^* = \sigma^{u_h}(p)|_{\llbracket 1, \ell'_k \rrbracket \times \llbracket 1, d_h \rrbracket}$ is a vertically aligned pattern whose projection $\widetilde w_h = \widetilde{\widetilde\pi}(w_h^*)$ is a word of the form  $a'_k$ or $a''_k$, and such that whenever $J_g \cap J_h \not= \varnothing$, $\widetilde w_g$ and $\widetilde w_h$ intersect at their initial and terminal segments, see Figure \ref{fig:diagram03}.

\begin{figure}[hbt]
\centering
\begin{tikzpicture}[scale=0.7]
\draw[line width = 0.5mm] (1,1) -- (6,1);
\draw[line width = 0.5mm] (1,1) -- (1,4) node[left]{$\ell'_k$} -- (1,6);
\draw[line width = 0.5mm] (1,6) -- (3,6) node[above]{$\ell_k'$} -- (6,6);
\draw[line width = 0.5mm] (6,1) -- (6,2) node[right]{$J_f$} -- (6,6);
\draw[line width = 0.5mm] (5.5,3) -- (10.5,3);
\draw[line width = 0.5mm] (5.5,3) -- (5.5,12) node[above , xshift=4mm]{$a_{k-1}^I$};
\draw[line width = 0.5mm] (10.5,3) -- (10.5,4) node[right]{$J_h$} -- (10.5,12);
\draw[line width = 0.5mm] (5.5,12) -- (10.5,12);
\draw[line width = 0.5mm] (5.5,7) -- (10.5,7);
\draw[line width = 0.5mm] (5.5,8) -- (10.5,8);
\draw[line width = 0.5mm] (9,6) -- (14,6);
\draw[line width = 0.5mm] (9,6) -- (9,11);
\draw[line width = 0.5mm] (9,11) -- (14,11);
\draw[line width = 0.5mm] (14,6) -- (14,7) node[right]{$J_g$} -- (14,11);
\draw[line width = 0.5mm, dashed] (5,1) node[below, xshift=4mm]{$a_{k-1}^T$} -- (5,6);
\draw[line width = 0.5mm, dashed] (6.5,3) -- (6.5,10) node[right]{$J_{h,1}$} -- (6.5,12);
\draw[line width = 0.5mm, dashed] (13,6) -- (13,11) node[above, xshift=4mm]{$a^T_{k-1}$};
\draw[pattern = {north east lines}, pattern color = black] (5,1) rectangle (6,6);
\draw[pattern = {north east lines}, pattern color = black] (13,6) rectangle (14,11);
\draw[pattern = {north west lines}, pattern color = black] (5.5,3) rectangle (6.5,12);
\filldraw (5.77,5) circle (1mm) node[right,xshift = 0.6mm]{$v$};
\draw[line width = 0.5mm, dashed] (0,3) node[left]{$u_h^y+1$} -- (5.5,3);
\draw[line width = 0.5mm, dashed] (0,12) node[left]{$u_h^y+d_h$} -- (5.5,12);
\end{tikzpicture}
\caption{We represent a clustering of two squares of the kind $J_h$ that intersects on the right $J_g$ along their markers and on the left $J_f$ along their initial and terminal segments.}
\label{fig:diagram03}
\end{figure}
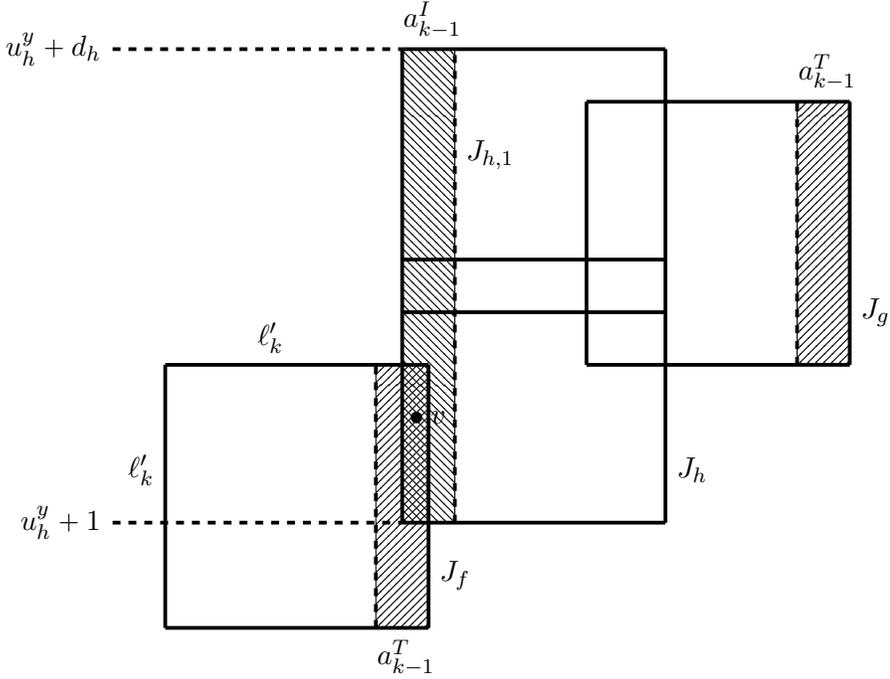
	
	We now show that an index $v=(v^x,v^y) \in J^A$ may belong to at most two rectangles $J_f$ and $J_h$. Indeed, by construction, as $u_g^x \not= u_h^x$, if $v^x$ belongs to two overlapping words of the form $a'_k,a''_k$,  then $v^x$ belongs to either  the intersection of the two markers $1^{(N'_k-1)\ell_{k-1}}$  or the intersection of the terminal segment $a_{k-1}^T$ of $a''_k$ and  the initial segment  $a_{k-1}^I$ of $a'_k$. In both cases described in Lemma~\ref{lemma:OverlappingWordsPrime} we exclude the overlapping of a third word of the form $ a'_k,a''_k$, thus we exclude the fact that $v$ may belong to a third rectangle $J_g$ with $u_g^x \not= u_f^x$ and $u_g^x \not= u_h^x$. Then
	\begin{align*}
	\Card(K^A) &= \sum_{v \in J^A} \mathds{1}_{(p(v)=0)} \\
	&\leq \sum_{h=1}^H \sum_{v \in (u_h+\llbracket 1, \ell'_k \rrbracket \times \llbracket 1, d_h \rrbracket)} \mathds{1}_{(p(v)=0)}  \leq \sum_{h=1}^H  f_{k-1}^A \ell_{k-1} d_h \\
	&\leq \frac{f_{k-1}^A\ell_{k-1}}{\ell'_k} \sum_{h=1}^H  \sum_{v \in J^A} \mathds{1}_{v \in (u_h+\llbracket 1, \ell'_k \rrbracket \times \llbracket 1,d_h \rrbracket)} =  \frac{f_{k-1}^A}{N'_k} \sum_{v \in J^A} \sum_{h=1}^H \mathds{1}_{(v \in J_h)} \\
	&\leq  \frac{2f_{k-1}^A}{N'_k} \Card(J ^A).
	\end{align*}
	This concludes our proof.\end{proof}

\section{Analysis of the zero-temperature limit}

\begin{proof}[Proof of Lemma \ref{lemma:BoundFromBellow}] Item \ref{item:BoundFromBellow_1}. Let $\mu$ be a measure satisfying $\Supp(\mu) \subseteq X_k=\langle L_k \rangle$ which is ergodic. Recall that Birkhoff's ergodic theorem extends to actions of countable amenable groups as long as the average is taken over a tempered F\o lner sequence~\cite{Lindenstrauss}. As the sequence $(\Lambda_n)_{n \in \NN}$ with $\Lambda_n =\llbracket -n,n\rrbracket ^d$ is tempered in $\ZZ^d$, it follows that for $\mu$-almost every point $x$
\[ \mu([\mathcal{F}]) =\lim_{n\to+\infty} \frac{ \Card(\{  u \in \Lambda_n : \sigma^u(x) \in [\mathcal{F}]\})}{\Card(\Lambda_n)}.
\]
We choose such a point $x \in \langle L_k \rangle$ and $s\in \llbracket 1, \ell_k \rrbracket^2$ such that $y=\sigma^{s}(x)$ and all its translates  $\sigma^{t \ell_k}(y)$,  $t \in \mathbb{Z}^2$,  satisfy $\sigma^{t \ell_k}(y|_{\llbracket 1, \ell_k \rrbracket^2}) \in L_k$. By taking a sub-sequence multiple of $\ell_k$ and by taking a set $\widetilde \Lambda_n$ tiled by translates of the square $\llbracket 1,\ell_k \rrbracket^2$,
\begin{gather*}
\widetilde\Lambda_n := \llbracket -n \ell_k, n \ell_k -1\rrbracket^2,
\end{gather*}
one obtains
\begin{align*}
\mu([\mathcal{F}]) &= \lim_{n\to+\infty} \frac{ \Card(\{  u \in \widetilde\Lambda_n -s: \sigma^u(y) \in [\mathcal{F}]\})}{\Card(\widetilde\Lambda_n)}, \\
&= \lim_{n\to+\infty} \frac{ \Card(\{  u \in \widetilde\Lambda_n : \sigma^u(y) \in [\mathcal{F}]\})}{\Card(\widetilde\Lambda_n)}.
\end{align*}
By definition of $L_k$, every pattern in $L_k$ is globally admissible and thus locally admissible,
\begin{gather*}
\forall\, t \in \llbracket -n, n-1\rrbracket^2, \ \forall\, v \in \llbracket 0, \ell_k-D \rrbracket^2, \   \sigma^{v+t \ell_k}(y)|_{\llbracket 1, D \rrbracket^2} \not\in \mathcal{[\mathcal{F}]}.
\end{gather*}
As $\Card( \llbracket 0, \ell_k-1 \rrbracket^2 \setminus \llbracket 0, \ell_k-D \rrbracket ^2) \leq  2D \ell_k$, we have
\begin{gather*}
\Card(\{  u \in \widetilde\Lambda_n : \sigma^u(y) \in [\mathcal{F}]\}) \leq  (2n)^2 2D \ell_k, \\
\Card(\widetilde\Lambda_n) = (2n)^2 \ell_k^2.
\end{gather*}
Therefore we get that $\mu([\mathcal{F}]) \leq 2D/\ell_k$.

Item \ref{item:BoundFromBellow_2}. Let  $\widetilde w \in \widetilde B_k$ be the word whose density of zeroes realizes the maximum value $f_k^B$. By Lemma \ref{lemma:OnedimensionalConcatenatedSubshift}, $\widetilde w$ is a subword of some $\widetilde x \in \widetilde X$. Let $\widetilde{\widetilde x}$ be the vertically aligned configuration corresponding to $\widetilde x$. By the simulation theorem, $\widetilde{\widetilde x} = \widehat\Pi(\widehat x)$ for some $\widehat x \in \widehat X$. Let $\widehat w  := \widehat x|_{\llbracket 1, \ell_k \rrbracket^2}$.  By duplicating the symbol $0$ we obtain,
\begin{gather*}
\Card(B_k) \geq \Card \big(\big\{ w \in \mathcal{A}^{\llbracket 1, \ell_k \rrbracket^2} : \Gamma(w) = \widehat w \big\}\big) = 2^{\ell_k^2f_k^B(\widetilde w)}, \\
h_{top}(X_k^B) = \frac{1}{\ell_k^2}\ln(\Card(B_k)) \geq \ln(2) f_k^B.
\end{gather*}

where $\Gamma$ has been defined in Equation \ref{equation_1}.

Item \ref{item:BoundFromBellow_3}. let $\mu_k^B$ be an ergodic measure of maximal entropy of $X_k^B$. Then 
\[
\Supp(\mu_k^B) \subseteq  X_k^B \ \ \text{and} \ \ P(\beta_k \varphi) \geq h(\mu_k^B)-\beta_k \mu_k^B([\mathcal{F}]) \geq \ln(2) f_k^B - 2D \frac{\beta_k}{\ell_k}.
\]
This is what we wanted to prove.\end{proof}

\begin{proof}[Proof of Lemma \ref{lemma:sizeNeighborhoodSubshift}]
As the pressure of $\beta_k \varphi$ is non-negative (the two configurations $1^\infty$ and $2^\infty$ belong to $X$ and $\varphi$ is identically zero on $X$), we have
\begin{gather*}
h(\mu_{\beta_k}) - \int \! \beta_k \varphi \, d\mu_{\beta_k} = P(\beta_k \varphi) \geq0, \quad
\mu_{\beta_k}([\mathcal{F}]) \leq  \frac{h(\mu_{\beta_k)}}{\beta_k} \leq \frac{ \ln(\Card(\mathcal{A})) }{\beta_k}, \\
\Sigma^2(\mathcal{A})  \setminus [M_k'] \subseteq \bigcup_{u \in  \llbracket 0,R'_k-D \rrbracket^2} \sigma^{-u}([\mathcal{F}]), \quad
\mu_{\beta_k}(\Sigma^2(\mathcal{A})  \setminus [M'_k]) \leq \frac{R'_k{}^2}{\beta_k}\ln( \Card(\mathcal{A})).
\end{gather*}\end{proof}

\begin{proof}[Proof of Lemma \ref{lemma:EntropyEstimate_1}]
By definition of relative entropy
\begin{align*}
H\Big(\mathcal{P}^{\llbracket 1,n \rrbracket^2},\mu_{\beta_k}\Big)
&= \dis H\Big(\mathcal{P}^{\llbracket 1,n \rrbracket^2}\bigvee {\widetilde{\mathcal{G}}}^{\llbracket 1,n \rrbracket^2} \bigvee \mathcal{U}_k^{\llbracket 0,n-R_k' \rrbracket^2},\mu_{\beta_k}\Big)\\
&=H\Big(\mathcal{P}^{\llbracket 1,n\rrbracket^2}\mid {\widetilde{\mathcal{G}}}^{\llbracket 1,n \rrbracket^2}\bigvee\mathcal{U}_k^{\llbracket 0,n-R_k' \rrbracket^2},\mu_{\beta_k}\Big)  \\
&\quad+ H \Big( {\widetilde{\mathcal{G}}}^{\llbracket 1,n \rrbracket^2} \mid \mathcal{U}_k^{\llbracket 0,n-R_k' \rrbracket^2}, \mu_{\beta_k} \Big)  + H \Big( \mathcal{U}_k^{\llbracket 0,n-R_k' \rrbracket^2}, \mu_{\beta_k} \Big).
\end{align*}
	
	The first term of the right-hand side is the relative  entropy at scale  $\ell'_k$ that requires a special treatment. The third term is computed using the estimate in lemma \ref{lemma:sizeNeighborhoodSubshift} (the function that maps $\epsilon \in (0,e^{-1})$ to $H(\epsilon)$ is  increasing),
\begin{align*}
H \big( \mathcal{U}_k^{\llbracket 0,n-R_k' \rrbracket^2}, \mu_{\beta_k} \big) 
&= \sum_{P\in \mathcal{U}_k^{\llbracket 0,n-R_k' \rrbracket^2}}-\mu_{\beta_k}(P)\ln(\mu_{\beta_k}(P)) \\
&\leq  n^2 H(\mathcal{U}_k, \mu_{\beta_k})  \leq  n^2 H(\epsilon_k).
\end{align*}
	
We now compute the term in the middle. We choose $\epsilon'_k > \epsilon_k$ and define
\begin{gather*}
U_n:=\Big\{ x \in \Sigma^2(\mathcal{A}):\Card\left\{u\in\llbracket0,n-R_k'\rrbracket^2:\sigma^u(x)\in[M_k']\right\}\geq n^2(1-\epsilon'_k )\Big\}.
\end{gather*}
	
By the $\ZZ^d$-version of Birkhoff's ergodic theorem we have that
\begin{displaymath}
\lim_{n \to +\infty} \mu_{\beta_k}(U_n) = 1.
\end{displaymath}
	
Let $(\mu_x)_{x \in \Sigma}$ be the family of conditional measures with respect to $\mathcal{U}_k^{\llbracket 0,n-R_k' \rrbracket}$. We have
\begin{align*}
H\left({\widetilde{\mathcal{G}}}^{\llbracket 1,n\rrbracket^2}\mid\mathcal{U}_k^{\llbracket 0,n-R'_k\rrbracket^2},\mu_{\beta_k}\right) 
&=  \int H\Big({\widetilde{\mathcal{G}}}^{\llbracket1,n\rrbracket^2},\mu_x \Big) d\mu_{\beta_k}(x) \\
&=\int_{U_n} H\Big({\widetilde{\mathcal{G}}}^{\llbracket1,n\rrbracket^2},\mu_x \Big) d\mu_{\beta_k}(x) + \\
&\quad +\int_{\Sigma^2(\mathcal{A})\setminus U_n} H\Big({\widetilde{\mathcal{G}}}^{\llbracket1,n\rrbracket^2},\mu_x \Big) d\mu_{\beta_k}(x) \\
&\leq  \dis \int_{U_n} H\Big({\widetilde{\mathcal{G}}}^{\llbracket 1,n \rrbracket^2},\mu_x\Big) d\mu_{\beta_k}(x) + \\
&\quad+ n^2 \mu_{\beta_k}\left(\Sigma^2(\mathcal{A}) \setminus U_n\right) \ln(\Card(\widetilde{\mathcal{A}})), 
\end{align*}
and therefore
\begin{displaymath}
\limsup_{n\to+\infty}\frac{1}{n^2}H \Big({\widetilde{\mathcal{G}}}^{\llbracket 1,n \rrbracket^2}\mid\mathcal{U}_k^{\llbracket 0,n-R_k' \rrbracket^2}, \mu_{\beta_k}\Big)\leq\limsup_{n\to+\infty}\int_{U_n}\frac{1}{n^2}H\Big({\widetilde{\mathcal{G}}}^{\llbracket 1,n \rrbracket^2},\mu_x\Big)d\mu_{\beta_k}(x).
\end{displaymath} 

We now consider a fixed $x \in U_n$. We compute the number of elements of ${\widetilde{\mathcal{G}}}^{\llbracket 1,n \rrbracket^2}$ that are compatible with the constraint
\begin{displaymath}
\Card\{u\in\llbracket 0,n-R_k'\rrbracket^2:\sigma^u(x)\in [M_k']\}\geq n^2 (1-\epsilon'_k ).
\end{displaymath}
Let $I$ be the subset
\begin{displaymath}
I :=\left\{u\in \llbracket 0,n-R_k'\rrbracket^2:\sigma^u(x)\in[M_k']\right\}.
\end{displaymath}
Since $x\in U_n$, we have
\begin{displaymath}
\frac{\Card(I)}{n^2}\geq 1-\epsilon_k'.
\end{displaymath}
Let $J \subseteq I$ be a maximal subset satisfying for every  $ u,v \in J$, 
\begin{displaymath}
\|u-v\|_\infty\geq\frac{1}{2}R_k'.
\end{displaymath}
For every $u\in J$, consider
\begin{displaymath}
I_u:=\Big\{v\in I:\|u-v\|_\infty<\frac{1}{2}R_k' \Big\}.
\end{displaymath}
By maximality of $J$ we have $I=\bigcup_{u \in J} I_u$. We first observe that the sets $\Big(u+\big\llbracket 1,\lceil R_k'/2\rceil\big\rrbracket^2\Big)_{u\in J}$ are pairwise disjoint. Then
\begin{displaymath}
\Card(J) \leq \frac{4n^2}{R_k'^2}.
\end{displaymath}
We also observe that for every $v_1,v_2\in I_u$, $\|v_1-v_2\|_\infty<R_k'$ and
\begin{displaymath}
\Big(v_1+\llbracket 1,R_k' \rrbracket^2 \Big) \bigcap \Big(v_2 + \llbracket 1,R_k' \rrbracket^2 \Big) \neq \varnothing.
\end{displaymath}
For every $u\in I$ let 
\begin{displaymath}
K_u: = \bigcup_{v \in I_u} \big(v+\llbracket 1,R_k' \rrbracket^2 \big)\subset \llbracket1,n\rrbracket^2.
\end{displaymath}
For every $v\in I_u$, we have 
\begin{displaymath}
x|_{v+ \llbracket 1,R_k' \rrbracket^2}\in [M_k'].
\end{displaymath}
In particular the pattern $x|_{v+ \llbracket 1,R_k' \rrbracket^2}$ is locally $\mathcal{F}$-admissible and satisfies the constraint that all the $\widetilde{\mathcal{A}}$-symbols are vertically aligned in $v+ \llbracket 1,R'_k \rrbracket^2$. Using that $K_u$ is connected as a Cayley subgraph of $\ZZ^2$ with the canonical generators, one obtains that the $\widetilde{\mathcal{A}}$-symbols are also vertically aligned  in $K_u$.
	
The width of $K_u$ is less than $2R_k'$, so the cardinality of possible patterns $p\in\widetilde{\mathcal{A}}^{K_u}$ satisfying the constraint that the $\widetilde{\mathcal{A}}$-symbols are vertically aligned is bounded by $\Card(\widetilde{\mathcal{A}})^{2R_k'}$. The cardinality of set of patterns with support $\dis \bigcup_{u\in J}K_u$ is thus bounded by 
\begin{displaymath}
\Big(\Card(\widetilde{\mathcal{A}})^{2R_k'}\Big)^{4n^2/R_k'^2}= \exp\left(\left[2R_k' \cdot \frac{4n^2}{R_k'^2}\right] \ln (\Card(\widetilde{\mathcal{A}}))\right) = \exp\left( \frac{8n^2}{R_k'} \ln (\Card(\widetilde{\mathcal{A}}))\right).
\end{displaymath}
Since $\dis \bigcup_{u\in J} K_u$ covers $I$, the cardinality of the set of patterns with support $\dis \llbracket 1,n \rrbracket^2\setminus\bigcup_{u\in J}K_u$ is bounded by $\Card(\widetilde{\mathcal{A}})^{n^2\epsilon_k'}$. We have proven that, for every $x \in U_n$,
\begin{displaymath}
H\Big({\widetilde{\mathcal{G}}}^{\llbracket1,n\rrbracket^2},\mu_x\Big)\leq\Big(\frac{8n^2}{R_k'}+n^2\epsilon_k' \Big)\ln(\Card(\widetilde{\mathcal{A}})).
\end{displaymath}
We conclude by letting $n\to+\infty$ and $\epsilon_k' \to \epsilon_k$.
\end{proof}

The proof of Lemma \ref{lemma:EntropyEstimate_2} requires the following intermediate result.

\begin{lemma}
\label{lemma:CoveringComplexity}
Let $n,\ell$ be integers which satisfy $n > 2 \ell >2$, $\epsilon \in(0,1)$, and let $S \subseteq \llbracket 0,n-2 \ell \rrbracket^2$ be a subset satisfying $\Card(S) \geq n^2(1-\epsilon)$. Let $\widehat E$ be the set
\begin{gather*}
\widehat E := \big\{ w \in \widehat{\mathcal{A}}^{\llbracket 1, n \rrbracket^2} : 
\forall\,  u \in S , \ \sigma^u(w)|_{\llbracket 1, 2 \ell\, \rrbracket^2} \in \mathcal{L}(\widehat X, 2\ell)  \big\}.
\end{gather*}
Then
\[
\frac{1}{n^2} \ln(\Card(\widehat E)) \leq \frac{1}{\ell} \ln(\Card(\widetilde{\mathcal{A}})) +  \frac{1}{\ell^2}\ln(C^{\widehat{X}}(\ell)) + \epsilon \ln (\Card(\widehat{\mathcal{A}})).
\]
\end{lemma}

\begin{proof}
To simplify the notations we assume that $n$ is a multiple of $\ell$. We decompose the square $\llbracket 1,n \rrbracket^2$ into a disjoint union of squares of size $\ell$,
\[ \llbracket1,n\rrbracket^2=\bigcup_{v \in \llbracket 0, \frac{n}{\ell} -1 \rrbracket^2} \left( \ell v+ \llbracket 1 ,\ell \rrbracket^2 \right). \]
	
We define the set of indexes $v$ that intersect $S$, more precisely, we have
\[ 
V:=\left\{v\in\Big\llbracket 0,\frac{n}{\ell}-2\Big\rrbracket^2:\left(\ell v+\llbracket 0,\ell -1\rrbracket^2\right)\bigcap S\neq \varnothing \right\}.
\]
	
Then for every  $w \in \widehat E$, $v \in V$, and $u \in \big( \ell v+ \llbracket 0 ,\ell-1 \rrbracket^2 \big) \bigcap S$, therefore
\[ \left(\ell v + \llbracket 1 + \ell , 2\ell \rrbracket^2 \right) \subseteq \left( u+\llbracket 1, 2\ell \rrbracket^2 \right). \]
Since we are taking $u\in S$ we have that
\[ \sigma^u(w) |_{\llbracket 1,2\ell \rrbracket^2} \in \mathcal{L}(\widehat X, 2\ell), \]
and then
\[ \sigma ^{\ell v+(\ell,\ell)}(w)|_{\llbracket 1 , \ell \rrbracket^2} \in \mathcal{L}(\widehat X, \ell). \]
	
The restriction of $w$ on every square $\big(\ell v + \llbracket 1 + \ell , 2\ell \rrbracket^2 \big)$ is globally admissible with respect to $\widehat{\mathcal{F}}$. Note that these squares are pairwise disjoint and the cardinality of their union  is at least $n^2(1-\epsilon)$, since
\[
\Card\left(\bigcup_{v\in V}\left(\ell v+\llbracket 1+\ell,2\ell\rrbracket^2\right)\right)=\Card\left(\bigcup_{v\in V}\left(\ell v+\llbracket 0 ,\ell- 1\rrbracket^2\right)\right)\geq\Card(S).
\]
	
Hence we proved that $\widehat E$ is a subset of the set of patterns $w$ made of independent and disjoint words $(w_v)_{v \in V}$, with $w_v \in \mathcal{L}(\widehat X,\ell)$, and  of arbitrary symbols on  $\llbracket 0, n-2\ell \rrbracket^2 \setminus S$ of size at most $\epsilon n^2$. Using the trivial bound $\Card(\mathcal{L}(\widetilde X, \ell)) \leq \Card(\widetilde{\mathcal{A}})^{\ell}$, we have
\[
\Card(\widehat E) \leq \left( \Card(\widetilde{\mathcal{A}})^{\ell} \cdot C^{\widehat{X}}(\ell) \right)^{\left(n/\ell\right)^2}\cdot \Card(\widehat{\mathcal{A}})^{\epsilon n^2}
\]
and therefore
\[
\frac{1}{n^2} \ln(\Card(\widehat E)) \leq \frac{1}{\ell} \ln(\Card(\widetilde{\mathcal{A}})) +  \frac{1}{\ell^2}\ln(C^{\widehat{X}}(\ell)) + \epsilon \ln (\Card(\widehat{\mathcal{A}})).
\]

\end{proof}

\begin{proof}[Proof of  Lemma \ref{lemma:EntropyEstimate_2}] \quad

{\it Proof of item \ref{item:EntropyEstimate_2_1}.} Using the $\ZZ^d$-version of Birkhoff's ergodic theorem and Lemma~\ref{lemma:sizeNeighborhoodSubshift}, it follows that for almost every $x\in\Sigma^2(\mathcal{A})$,
\[
\lim_{n\to+\infty} \frac{1}{n^2} \Card \big(\big\{ u \in \llbracket 0,n-R'_k \rrbracket^2 : \sigma^u(x) \in [M'_k] \big\} \big) = \mu_{\beta_k}([M'_k])
\]
and
\[
\lim_{n \to+\infty} \frac{1}{n^2} \Card \big( \big\{ u \in \llbracket 1,n \rrbracket^2 : \pi(x(u))=\widetilde a \big\} \big) = \mu_{\beta_k}(\widetilde G_{\widetilde a}), \quad \forall\, \widetilde a \in \widetilde{\mathcal{A}}.
\]

We choose $n> R'_k$. An element  of the partition $\widetilde{\mathcal{G}}^{\llbracket 1, n \rrbracket^2} \bigvee \mathcal{U}^{\llbracket 0,n-R'_k \llbracket^2}$ is of the form $\widetilde G_p \cap U_S$ where $p \in \widetilde{\mathcal{A}}^{\llbracket 1, n \rrbracket^2}$ is a pattern and $S \subseteq \llbracket 0,n-R'_k\rrbracket^2$ is a subset, where
\begin{gather*}
U_S := \big\{ x \in \Sigma^2(\mathcal{A}) : \forall\,  u \in S, \ \sigma^u(x) \in [M'_k], \ \forall\, u \in \llbracket 0,n-R'_k\rrbracket^2 \setminus S, \ \sigma^u(x) \not\in [M'_k] \big\}, \\
\widetilde G_p := \{ x \in \Sigma^2(\mathcal{A}) : \Pi(x|_{\llbracket 1, n \rrbracket^2}) = p \}.
\end{gather*}
Let $\eta < \mu_{\beta_k}(\widetilde G_0)$.  Lemma~\ref{lemma:sizeNeighborhoodSubshift} implies $\mu_{\beta_k} \big(\Sigma^2(\mathcal{A}) \setminus [M'_k] \big) \leq \epsilon_k$. Using again Birkhoff's theorem we get
\begin{gather*}
 \lim_{n\to+\infty} \frac{1}{n^2} \Card \left(\{  u \in \llbracket 0, n-R'_k \rrbracket : \sigma^u(x) \in [M'_k] \}\right) > 1-\epsilon_k, \text{for $\mu_{\beta_k}$-a.e. $x$,}
 \end{gather*}
 \begin{gather}\label{equation_3}
\lim_{n\to+\infty} \mu_{\beta_k} \Big( \bigcup\nolimits_{S \subseteq \llbracket 0,n-R'_k\rrbracket^2} \left\{ U_S :  \Card(S) > n^2(1-\epsilon_k) \right\} \Big) = 1.
\end{gather}
For $n$ large enough, we choose $S\subseteq\llbracket 0,n-R_k'\rrbracket^2$ such that $U_S \not= \varnothing$ and $\Card(S)\geq n^2(1-\epsilon_k)$. By definition of $M'_k$ and $T'_k$ (see page 16), if $x\in U_S$, then for every $u\in S$, $\sigma^u(x)|_{\llbracket 1,R'_k \rrbracket^2}$ is a locally admissible pattern with respect to $\mathcal{F}$ and
\[ \sigma^{u+T'_k}(x)|_{\llbracket 1,2\ell'_k\rrbracket^2} \in \mathcal{L}(X,2\ell'_k). \]
Define for every $n>R_k'$ and every pattern $p \in \widetilde{\mathcal{A}}^{\llbracket 1, n \rrbracket^2}$ the set
\[ K_n(p):=\{u\in\llbracket 1,n\rrbracket^2:p(u)=0\}. \]
As $\mu_{\beta_k}(\widetilde G_0)>\eta$, it follows by Birkhoff's ergodic theorem
 \begin{gather}\label{equation_4}
\lim_{n\to+\infty}\mu_{\beta_k}\left(\bigcup\nolimits_p\left\{\widetilde G_p:\Card(K_n(p))>n^2  \eta\right\} \right)=1. 
 \end{gather}
From Equations \ref{equation_3} and \ref{equation_4}, for large $n$, one can choose $S$ and $p$ such that $U_S \cap \widetilde G_p \not= \varnothing$,  $\Card(K_p) > n^2 \eta$ and $\Card(S) \geq n^2(1-\epsilon_k)$. Using the notations in Definition \ref{Definition:IJK} and the conclusions of Lemma \ref{lemma:AdmissibilityIntermediateScale}, one obtains
\[ T'_k+S \subseteq I(p,\ell'_k) \quad \mbox{and} \quad \tau'_k+I(p,\ell'_k)\subseteq J^A(p,\ell'_k)\sqcup J^B(p,\ell'_k)=:J^A\sqcup J^B, \]
therefore by our choice of $S$ we obtain
\begin{equation}
\label{eq.Lemma15.1}
n^2(1-\epsilon_k)\leq \Card(S) = \Card(\tau'_k+T'_k+S) \leq \Card \big(J^A \sqcup J^B \big) \leq n^2.
\end{equation}
Besides that, we have
\[ n^2\eta \leq \Card(K_n(p)) \leq \Card(K^A \sqcup K^B)+n^2 \epsilon_k \]
and by the Lemma~\ref{lemma:FrequencyOf0} we have
\[ \Card(K_n(p)) \leq  \frac{2}{N'_k}\Card(J^A)f_{k-1}^A +  \big(1-N_{k-1}^{-1}\big)^{-1} \Card(J^B) f_{k-1}^B + n^2 \epsilon_k. \]

We divide each term by $n^2$ and take the limit with $n\to+\infty$, $\epsilon \to \epsilon_k$, and $\eta \to \mu_{\beta_k}([0])$.

{\it Proof of item \ref{item:EntropyEstimate_2_2}.}
We now assume that $k$ is even and choose $\eta$ such that $\mu_{\beta_k}(\widetilde G_1) > \eta$. We choose $p\in\widetilde{\mathcal{A}}^{\llbracket1,n\rrbracket^2}$ such that $\widetilde G_p \cap U_S \not= \varnothing$ and
\begin{equation}
\label{eq.Lemma15.2}
\dis \Card \left( \left\{ u \in \llbracket 1,n\rrbracket^2 : p(u)=1 \right\} \right) > \eta  n^2.
\end{equation}
Let $x \in \widetilde G_p \cap U_S$ and $(\mu_x)_{x \in\Sigma}$ be the family of conditional measures with respect to the partition $\widetilde{\mathcal{G}}^{\llbracket 1, n \rrbracket^2} \bigvee \mathcal{U}^{\llbracket 0,n-R_k \llbracket^2}$. We use the trivial upper bound of the entropy, so
\begin{equation}
\label{eq.Lemma15.6}
\dis H( \mathcal{G}^{\llbracket 1,n \rrbracket^2}, \mu_x) \leq \ln(\Card(E_{p,S}))
\end{equation}
where
\[
E_{p,S} := \big\{ w \in \mathcal{A}^{\llbracket 1,n \rrbracket^2} :  \Pi(w)=p \ \ \text{and} \ \ 
\forall u \in S, \ \sigma^{u+T'_k}(w)|_{\llbracket 1,2\ell'_k\rrbracket^2} \in \mathcal{L}(X,2\ell'_k) \big\}.
\]
Also consider
\[ \widehat E_{p,S} := \Gamma(E_{p,S}). \]

Note that every pattern in $E_{p,S}$ is obtained from a pattern in $\widehat E_{p,S}$ by duplicating the symbol $0$. Using Lemma~\ref{lemma:CoveringComplexity} we conclude that
\begin{gather*}
\ln(\Card(E_{p,S}) ) \leq \ln( \Card(\widehat E_{p,S}) ) + \Card( K_n(p)) \ln(2), \\
\frac{1}{n^2} \ln(\Card(\widehat E_{p,S})) \leq \frac{1}{\ell'_k} \ln(\Card(\widetilde{\mathcal{A}})) +  \frac{1}{\ell'_k{}^2}\ln(C'_k) + \epsilon_k \ln (\Card(\widehat{\mathcal{A}})),
\end{gather*}
thus
\begin{equation}
\label{eq.Lemma15.3}
\begin{array}{rcl}
\dis \frac{1}{n^2} \ln(\Card(E_{p,S})) & \leq &  \dis \Big( \frac{2}{N'_k}\Card(J^A) f_{k-1}^A + (1-N_{k-1}^{-1})^{-1} \Card(J^B) f^B_{k-1}+ n^2 \epsilon_k \Big) \frac{\ln(2)}{n^2} + \\
   &   & \dis + \frac{1}{\ell'_k} \ln(\Card(\widetilde{\mathcal{A}})) +  \frac{1}{\ell'_k{}^2}\ln(C_k') + \epsilon_k \ln (\Card(\widehat{\mathcal{A}})).
\end{array}
\end{equation}

The symbol $1$ does not appear in $J^B=J^B(p,\ell'_k)$, thus
\[
\big\{ u \in \llbracket 1,n\rrbracket^2 : p(u)=1 \big\} \subset J^A  \sqcup \left( \llbracket 1,n \rrbracket^2 \setminus (J^A\sqcup J^B) \right).
\]
Since we are assuming (\ref{eq.Lemma15.2}), using (\ref{eq.Lemma15.1}) and the fact that $J^A \cap J^B = \emptyset$, we obtain
\begin{equation*}
\Card \left(\llbracket 1,n\rrbracket^2  \setminus (J^A\sqcup J^B) \right) \leq  n^2\varepsilon_k
\end{equation*}
\begin{equation}
\label{eq.Lemma15.4}
\dis \Card(J^A) \geq n^2 \Big( \eta - \epsilon_k \Big) \quad\mbox{and}\quad \Card(J^B) \leq n^2 \Big(1-\eta+\epsilon_k \Big).
\end{equation}

By replacing the upper bound for $\Card(J^B)$ given in (\ref{eq.Lemma15.4}) and $\Card(J^A)\leq n^2$ in (\ref{eq.Lemma15.3}) we obtain that
\begin{equation}
\label{eq.Lemma15.5}
\begin{array}{rcl}
\dis \frac{1}{n^2} \ln(\Card(E_{p,S})) & \leq &  \dis \left( \frac{2}{N'_k}f_{k-1}^A + (1-N_{k-1}^{-1})^{-1}\left( 1-\eta +\epsilon_k \right) f^B_{k-1}+ \epsilon_k \right)\ln(2) + \\
&   & \dis + \frac{1}{\ell'_k} \ln(\Card(\widetilde{\mathcal{A}})) +  \frac{1}{\ell'_k{}^2}\ln(C_k') + \epsilon_k \ln (\Card(\widehat{\mathcal{A}})).
\end{array}
\end{equation}
By integrating with respect to $\mu_{\beta_k}$ in both sides and taking the limit when $n\to+\infty$ we obtain item (2) of Lemma \ref{lemma:EntropyEstimate_2}. Item (3) is analogous.
\end{proof}

\section{Results on computability}
\label{section.computability}

In this last section we prove the two bounds on the relative complexity and reconstruction functions. 
The subshift of finite type $\widehat X=\Sigma^2(\widehat{\mathcal{A}},\widehat{\mathcal{F}})$ in the Aubrun-Sablik construction~\cite{AubrunSablik2013} as described in Theorem~\ref{theorem:AubrunSablikSimulation} is composed of four layers, that is, it is a subshift of a product of four subshifts of finite type, which is itself described by a finite set of forbidden patterns which impose conditions on how the layers superpose. See Figure 14 of~\cite{AubrunSablik2013}. The layers are:

\begin{enumerate}
	\item \texttt{Layer} 1: The set of all configurations $x\in \widetilde{\mathcal{A}}^{\ZZ^2}$ that are vertically aligned, that is, $x_{u} = x_{u+(0,1)}$ for every $u \in \ZZ^2$.
	\item \texttt{Layer} 2: $\textbf{T}_{\texttt{Grid}}$ A subshift of finite type extension of a sofic subshift which is generated by the substitution given in Figure 3 of~\cite{AubrunSablik2013}. The sofic subshift induces infinite vertical ``strips'' of computation which are of width $2^n$ for every $n \in \NN$ and occur with bounded gaps (horizontally) in any configuration. It also encodes a ``clock'' on every computation strip of width $2^n$, which counts and restarts periodically every $2^{2^n}+2$ vertical steps in its strip.
	\item \texttt{Layer} 3: $\mathcal{M}_{\texttt{Forbid}}$ A subshift of finite type given by Wang tiles which replicates, on top of each clock determined by the previous layer, the space-time diagram of a Turing machine which enumerates all forbidden patterns of the effective subshift $\widetilde{X}$. It also communicates information from the space-time diagram to the fourth layer.
	\item \texttt{Layer} 4: $\mathcal{M}_{\texttt{Search}}$ A subshift of finite type given by Wang tiles which simulates a Turing machine which serves the purpose of checking whether the patterns enumerated by the third layer appear in the first layer. Each machine searches for forbidden patterns in a ``responsibility zone'' which is determined by the hierarchical structure of \texttt{Layer} 2.
\end{enumerate}

The rules between the four layers described in~\cite{AubrunSablik2013} force the Turing machine space-time diagrams to occur in every strip, and to restart their computation after an exponential number of steps. This ensures that every configuration witnesses every step of computation in a relatively dense set, and that every forbidden pattern is written on the tape by the Turing machine in every large enough strip. The fourth layer uses the information from the third layer to search for occurrences of the forbidden patterns in the first layer and thus discards any configuration in the first layer where one of these patterns occurs.

In the proofs that follow, we shall use nomenclature of~\cite{AubrunSablik2013} and refer to explicit parts and bounds associated to their construction, so the reader should bear in mind that this section is not meant to be self-contained. However, we shall aim to explain our arguments in such a way that at least they can be understood intuitively without the need to refer to~\cite{AubrunSablik2013}.

\begin{proof}[Proof of Proposition~\ref{proposition:ComplexityFunctionEstimate}]
	Let us denote by $C_n(\texttt{Layer}_k(\widehat{X}))$ the complexity of the projection to the $k$-th layer. and by $C_n(\texttt{Layer}_k(\widehat{X}) | \texttt{Layer}_j(\widehat{X})  )$ the complexity of the projection to the $k$-th layer given that there is a fixed pattern on the $j$-th layer. Clearly we have that 
	\begin{multline*}
	C^{\widehat{X}}(n) \leq C_n(\texttt{Layer}_1(\widehat{X})) \cdot C_n(\texttt{Layer}_2(\widehat{X})) \cdot C_n(\texttt{Layer}_3(\widehat{X})| \texttt{Layer}_2(\widehat{X}) ) \cdot \\ \cdot C_n(\texttt{Layer}_4(\widehat{X})|\texttt{Layer}_2(\widehat{X})).
	\end{multline*}
	
	\begin{itemize}
		\item \texttt{Layer} 1: As this layer is given by all $x\in \widetilde{\mathcal{A}}^{\ZZ^2}$ so that $x_{u} = x_{u+(0,1)}$ for every $u \in \ZZ^2$, a trivial upper bound for the complexity is \[ C_n(\texttt{Layer}_1(\widehat{X})) = \mathcal{O}(|\widetilde{\mathcal{A}}|^n).   \] 
		
		In fact, as in the end the only configurations which are allowed are those whose horizontal projection lies in the effective subshift $\widetilde{X}$, a better bound is given by $C_n(\texttt{Layer}_1(\widehat{X})) = \mathcal{O}(\exp( n\ h_{\mbox{top}}(\widehat{X})))$. For simplicity, we shall just keep the trivial bound.
		
		\item \texttt{Layer} 2: The complexity of every substitutive subshift in $\ZZ^2$ is $\mathcal{O}(n^2)$. To see this, suppose that the substitution sends symbols of some alphabet $\mathcal{A}_2$ to $n_1\times n_2$ arrays of symbols. By definition, every pattern of size $n$ occurs in a power of the substitution. If $k$ is such that $\min\{n_1,n_2\}^{k-1} \leq n \leq \min\{n_1,n_2\}^{k}$, then necessarily any pattern of size $n$ occurs in the concatenation of at most $4$ $k$-powers of the substitution. There are $|\mathcal{A}_2|^4$ choices for the $k$-powers and at most $(\max\{n_1,n_2\}^k)^2\leq (n\max\{n_1,n_2\})^2$ choices for the position of the pattern. It follows that there are at most $(|\mathcal{A}_2|^4\max\{n_1,n_2\}^2) n^2 = \mathcal{O}(n^2)$ patterns of size $n$. Furthermore, Mozes construction~\cite{Mozes1989} of an SFT extension for substitutive subshifts does not increase the complexity by more than a constant. We obtain,
		\[ C_n(\texttt{Layer}_2(\widehat{X})) = \mathcal{O}(n^2)\]
		
		The reader can find further information on dynamical systems generated by subtitutions in~\cite{Mozes1989,PytheasFogg}.
		
		\item \texttt{Layer} 3: It can be checked directly from the Aubrun-Sablik construction that the symbols on the third layer satisfy the following property: if the symbols on the substitution layer are fixed, then for every $u \in \ZZ^2$ the symbol at position $u$ is uniquely determined by the symbols at positions $u-(0,1), u-(1,1)$ and $u-(-1,1)$.  In consequence, it follows that knowing the symbols at positions in the ``U shaped region'' \[U = ( \{0\} \times \llbracket 1,n-1 \rrbracket ) \cup (\llbracket 0,n-1 \rrbracket\times \{0\}) \cup (\{n-1 \} \times \llbracket 1,n-1 \rrbracket) \]  completely determines the pattern. Therefore, if this layer has an alphabet $\mathcal{A}_3$, we have \[
		C_n(\texttt{Layer}_3(\widehat{X})| \texttt{Layer}_2(\widehat{X}) ) \leq |\mathcal{A}_3|^{3n-2} \leq \mathcal{O}(K_1^n),   \]
		for some positive integer $K_1$.
		
		\item \texttt{Layer} 4: The same argument for Layer $3$ holds for Layer 4. Therefore, if the alphabet of layer $4$ is $\mathcal{A}_4$ we have that for some positive integer $K_2$,
		\[
		C_n(\texttt{Layer}_4(\widehat{X})| \texttt{Layer}_2(\widehat{X}) ) \leq |\mathcal{A}_4|^{3n-2} \leq \mathcal{O}(K_2^n).   \]
	\end{itemize}
	
	Putting the previous bounds together, we conclude that there is some constant $K>0$ such that
	\[C^{\widehat{X}}(n) = \mathcal{O}(n^2 K^n).\]	
	This yields the desired bound on Proposition~\ref{proposition:ComplexityFunctionEstimate}.\end{proof}

We proved in Lemma \ref{proposition:algorithm} that $\widetilde{\mathcal{F}}$ satisfies the assumptions of Proposition \ref{proposition:ReconstructionFunctionEstimate}. We now prove the upper bound of the reconstruction function $R^{\widehat{X}} \colon \NN\to \NN$ of $\widehat{X}= \Sigma^2(\widehat{\mathcal{A}},\widehat{\mathcal{F}})$. Of course, a formal proof of these estimates would require a restatement of the construction of Aubrun-Sablik with all its details, which is out of the scope of this paper. Instead, we shall argue that the bounds we give suffice, making reference to the properties of the Aubrun-Sablik construction.

A description of $\widehat{\mathcal{F}}$ is given in~\cite{AubrunSablik2013} in an (almost) explicit manner for all layers except the substitution layer. For the substitution layer, a description of the forbidden patterns can be extracted in an explicit way from the article of Mozes~\cite{Mozes1989}.

The behavior of layers 2,3 and 4 in the Aubrun-Sablik construction is mostly independent of layer 1, except for the detection of forbidden patterns which leads to the forbidden halting state of the machine in the third layer. Because of that reason the analysis of the reconstruction function $R^{\widehat{X}}$ can be split into two parts:
\begin{enumerate}
	\item \textbf{Structural: } Assuming that the contents of the first layer are globally admissible (the configuration in the first layer is an extension of a configuration from $\widetilde{X}$), we give a bound that ensures that the contents of layers $2,3$ and $4$ are globally admissible, that is:
	\begin{itemize}
		\item The contents of layer $2$ correspond to a globally admissible pattern in the substitutive subshift and the clock.
		\item The contents of layer $3$ and $4$ correspond to valid  space-time diagrams of Turing machines that correctly align with the clocks.
	\end{itemize}
	\item \textbf{Recursive: } A bound that ensures that the contents of the first layer are globally admissible. This bound will of course depend upon $R^{\widetilde{X}}$ and $T^{\widetilde X}$.
\end{enumerate}

Finally, we are able to prove the upper bound for the reconstruction function given by Proposition~\ref{proposition:ReconstructionFunctionEstimate}.

\begin{proof}[Proof of Proposition~\ref{proposition:ReconstructionFunctionEstimate}]
Let us begin with the structural part, as it is simpler and does not depend upon $\widetilde{X}$. Let $p$ be a pattern with support $B_n$ and assume that the first layer of $p$ is globally admissible.
	
From Mozes's construction of SFT extensions for substitutions~\cite{Mozes1989} it can be checked that any locally admissible pattern of support $B_n$ of Mozes's SFT extension of a primitive substitution (The Aubrun-Sablik substitution is primitive) is automatically globally admissible. Let $k$ be the smallest positive integer such that the second layer of $p$ occurs within four level $k$ macrotiles of the substitution (each has size $4^k \times 2^k$) in any locally admissible pattern of that support.
	
Next, a clock runs on every strip of the Aubrun-Sablik construction. By the previous argument, the largest zone which intersects $p$ in more than one position is of level at most $k$. Therefore its largest computation strip has horizontal length $2^k$. In order to ensure that the clock starts on a correct configuration on every strip contained in $p$, we need to witness this pattern inside a locally admissible pattern which stacks $2^{2^k}+2$ macrotiles of level $k$ vertically. This ensures that the clock tiles occurring in $p$ are globally admissible.
	
Finally, if every clock occurring in $p$ starts somewhere, then the contents of the third layer are automatically correct, as they are determined every time it restarts. To certify that the fourth layer restricted to $p$ is globally admissible, we just need extend the horizontal length of our pattern twice, so that the responsibility zone of the largest strip is contained in it. 
	
By the previous arguments, it would suffice to witness $p$ inside a locally admissible pattern which contains in its center a $4 \times (2^{2^k}+2)$ array of macrotiles, each one of size $4^k \times 2^k$. Clearly the vertical term is dominant, thus it suffices, for large enough $n$, to fix a square box of size length $2^{k}(2^{2^k}+2)$.

As $2\cdot 4^{k-1} < n \leq 2\cdot 4^{k}$, there is a constant $C_0 >0$ such that an estimate for this part of the reconstruction function can be written as \[ R_{\mathrm{Struct}}^{\widehat{X}}(n) = \mathcal{O}(\sqrt{n}C_0^{\sqrt{n}}).  \]

Let us now deal with the recursive part. We need to find a bound such that the word of length $n$ occurring in the first layer of $p$ is globally admissible. By definition of $R^{\widetilde{X}}$, it suffices to have $p$ inside a pattern with support $B_{R^{\widetilde{X}}(n)}$ and check that the first layer is locally admissible with respect to $\widetilde{\mathcal{F}}$. In other words, we need to have the Turing machine $\widetilde{\mathbb{M}}$ check all forbidden words of length $R^{\widetilde{X}}(n)$ in this pattern. Luckily, the number of steps in order to do this is already computed in Aubrun and Sablik's article. After Fact 4.3 of~\cite{AubrunSablik2013} they show that, if $p_0,p_1,\dots,p_r$ are the first $r+1$ patterns enumerated by the Turing machine $\widetilde{\mathbb{M}}$, then the number of steps $S(p_0,\dots,p_r)$ needed in a computation zone to completely check whether a pattern from $\{p_0,\dots,p_r\}$ occurs in its responsibility zone of level $m$ satisfies the bound,\[ S(p_0,\dots,p_r) \leq T(p_0,\dots,p_r) + (r+1)\max(|p_0|,\dots,|p_r|)m^22^{3m+1},   \]
where $T(p_0,\dots,p_r)$ is the number of steps needed by $\mathcal{M}$ to enumerate the patterns $p_0,p_1,\dots,p_r$.
	
Recall that the assumptions of Proposition~\ref{proposition:ReconstructionFunctionEstimate} are that $R^{\widetilde{X}}(n)\leq Cn$ for some constant $C_1>0$ and that that the time enumeration function satisfies $T(R^{\widetilde{X}}(n)) \leq P(n)|\widetilde{\mathcal{A}}|^{n}$ for some polynomial $P$. In our construction, we may rewrite the Aubrun-Sablik formula so that the number $S(R^{\widetilde{X}}(n))$ of steps needed to check that all forbidden patterns of length at most $R^{\widetilde{X}}(n)$ in a responsibility zone of level $m$ satisfies the bound \begin{align*}
S(R^{\widetilde{X}}(n)) & \leq T(R^{\widetilde{X}}(n)) + |\widetilde{\mathcal{A}}|^{R^{\widetilde{X}}(n)+1}R^{\widetilde{X}}(n)m^22^{3m+1}\\
&  \leq P(C_1n)|\widetilde{\mathcal{A}}|^{n} + |\widetilde{\mathcal{A}}|^{C_1n + 1}C_1 n m^22^{3m+1}
\end{align*}
Simplifying the above bound, it follows that there exist constants $C_2,C_3>0$ such that
\[ S(R^{\widetilde{X}}(n)) \leq C_2 m^22^{3m+C_3n}.    \]
	
As $n$ is fixed, it follows that there is a smallest $\widebar{m}=\widebar{m}(n) \in \NN$ such that $2^{\widebar{m}} \geq C_4n$ (in such a way that the tape on the computation zone of level $\widebar{m}$ can hold words of size $R^{\widetilde{X}}(n)$) and such that
\[C_2 \widebar{m}^22^{3\widebar{m}+C_3n} \leq 2^{2^{\widebar{m}}}+2,\]
so that the number $2^{2^{\widebar{m}}}+2$ of computation steps in the zone of level $\widebar{m}$ is enough to check all the words of size $R^{\widetilde{X}}(n)$. It follows that a bound for the recursive part of $R^{\widehat{X}}$ is given by
\[ R_{\mathrm{recursive}}^{\widehat{X}}(n) = \mathcal{O}(2^{\widebar{m}+2^{\widebar{m}(n)}}).   \]
	
In order to turn this into an explicit asymptotic expression we need to find a suitable bound for $\widebar{m}(n)$. Notice that if $m \geq 6$ we simultaneously have that $m^2 \leq 2^m$ and $4m \leq 2^{m-1}$. We may then write for $m \geq 6$,
\[ C_2 m^22^{3m+C_3n} \leq C_2 2^{4m+C_3n} \leq C_2 2^{C_3n}2^{2^{m-1}}+2.   \]
Therefore, it suffices to find $\widebar{m} =\widebar{m}(n)$ such that
\[ C_22^{C_3n} \leq 2^{2^{\widebar{m}-1}}.   \]
From here, it follows that there is a constant $C_5 >0$ such that any value of $\widebar{m}$ satisfying
\[ \widebar{m} \geq C_5 + \log_2(n),    \]
satisfies the above bound. We get that
\[ R_{\mathrm{recursive}}^{\widehat{X}}(n) = \mathcal{O}(n 2^{C_5n}).  \]
	
Finally, putting together the structural and recursive asymptotic bounds, we obtain that there is a constant $K >0$ such that
\[  R^{\widehat{X}}(n) = \mathcal{O}(\max \{ \sqrt{n}C_0^{\sqrt{n}} , n 2^{C_5n}\}) = \mathcal{O}(nK^n).    \]
Hence we get that \[ \limsup_{n\to+\infty}\frac{1}{n}\log( R^{\widehat{X}}(n) )  < +\infty. \]
This is what we wanted to prove.\end{proof}

{\it Acknowledgments:} We are very grateful to Nathalie Aubrun and Mathieu Sablik for always being available to discuss and explain details of their construction to us and also to Samuel Petite for his help and suggestions. RB, GDV, and PT were partially supported by USP-COFECUB Uc Ma 176/19,
"{\it Formalisme Thermodynamique des quasi-cristaux \`a temp\'erature z\'ero}". GDV was financed in part by the Coordenação de Aperfeiçoamento de Pessoal de Nível Superior - Brasil (CAPES) - Finance Code 001 (Project 88887.1351235/2017-01). S. Barbieri, who collaborated in the computer-theoretic parts of the paper, was supported by the FONDECYT grant 11200037 and the ANR projects CoCoGro (ANR-16-CE40-0005) and CODYS (ANR-18-CE40-0007). RB was supported by CNPq grants 312294/2018-2 and 408851/2018-0, by FAPESP grant 16/25053-8, and by the Grant for experienced researchers from abroad of the Centre of Excellence “Dynamics, Mathematical Analysis and Artificial Intelligence” at Nicolaus Copernicus University in Toruń, Poland. We thank both anonymous referees who provided useful and detailed comments on a previous version of this manuscript.

\end{document}